\documentclass[10pt]{article}

\usepackage{amssymb,amsthm,bbm}
\usepackage[centertags]{amsmath}
\usepackage{mathabx}
\usepackage{MnSymbol}
\usepackage[mathscr]{euscript}
\usepackage{tikz-cd}
\usepackage{hyperref}

\newtheorem{thm}{Theorem}[subsection]
\newtheorem{theorem}[thm]{Theorem}

\newtheorem{lemma}[thm]{Lemma}

\newtheorem{proposition}[thm]{Proposition}
\newtheorem{cor}[thm]{Corollary}

\theoremstyle{definition}

\newtheorem{definition}[thm]{Definition}

\newtheorem{example}[thm]{Example}

\newtheorem{remark}[thm]{Remark}

\newtheorem{construction}[thm]{Construction}

\newcommand{\Fin}{\mathscr{F}in}
\newcommand{\Snexc}{[\mathscr{C}, \mathscr{S}]^{(n)}}

\newcommand{\exc}[1]{^{(#1)}}		% index for n-excisive functors

\newcommand{\slice}[1]{_{/#1}}	% notation for slice categories

\newcommand{\gap}[2]{(#1, #2)}                 % The cartesian gap map
\newcommand{\cogap}[2]{\lfloor #1, #2 \rfloor} % The cocartesian gap map

\usepackage{rotating}
\newcommand{\sfGamma}{\mathsf{\Gamma}}

\newcommand{\cogapcube}[1]{\sfGamma(#1)}   % The cocartesian gap map for cubes

\newcommand{\pp}{\square}             % The pushout product (mat's shortcut)

\newcommand{\ph}[2]{\langle #1, #2 \rangle}                % The pullback hom (mat's shortcut)
\newcommand{\intph}[2]{\llangle #1, #2 \rrangle}           % The internal pullback hom (mat's shortcut)

\newcommand{\fwd}[2]{\{#1,#2\}}                            % The fiberwise diagonal

\renewcommand{\lim}{\operatorname*{lim}}%

\newcommand{\lto}{\leftarrow}				% (mat)
\newcommand{\ini}{0}				% command for initial object	(mat)
\newcommand{\term}{1}				% command for terminal object		(mat)

\newcommand{\intperp}{{\begin{sideways}$\!\Vdash$\end{sideways}}}		% Internal orthogonality (mat)
\newcommand{\iperp}{\ \intperp\ }								% Internal orthogonality with space on sides (mat)
\newcommand{\relintperp}[1]{\intperp_{#1}}						% Relative internal orthogonality (mat)
\newcommand{\relperp}[1]{\perp_{#1}}							% Relative internal orthogonality (mat)
\newcommand{\reliperp}[1]{\ \intperp_{#1}\ }						% Relative internal orthogonality with spaces on sides (mat)
\newcommand{\sperp}{\upModels}								% Parametric orthogonality

\newcommand{\map}[2]{[#1, #2]}

\newcommand{\intmap}[2]{\lsem #1, #2 \rsem}
\newcommand{\relintmap}[3]{\lsem #2, #3 \rsem_{#1}}

\newcommand{\scr}[1]{\ensuremath{\mathscr{#1}}}
\newcommand{\Cat}[1]{\ensuremath{\mathscr{#1}}}

\newcommand{\pbmark}{\ar[dr, phantom, "\ulcorner" very near start, shift right=1ex]}
\newcommand{\pomark}{\ar[ul, phantom, "\lrcorner" very near start, shift right=1ex]}

\newcommand{\ul}{\underline}
\newcommand{\op}{^\mathrm{op}}
\newcommand{\join}{\star}

\newcommand{\surj}{\twoheadrightarrow}

\newcommand{\tto}{\longrightarrow}

\DeclareMathOperator*{\holim}{lim}
\DeclareMathOperator*{\colim}{colim}
\DeclareMathOperator{\fib}{fib}
\DeclareMathOperator{\id}{id}

\newcommand{\cC}{\mathscr{C}}

\newcommand{\cE}{\mathscr{E}}
\newcommand{\cF}{\mathscr{F}}

\newcommand{\cK}{\mathscr{K}}
\newcommand{\cL}{\mathscr{L}}
\newcommand{\cM}{\mathscr{M}}

\newcommand{\cR}{\mathscr{R}}
\newcommand{\cS}{\mathscr{S}}
\newcommand{\cT}{\mathscr{T}}

\newcommand{\cX}{\mathscr{X}}

\makeatletter
\pgfqkeys{/tikz/commutative diagrams}{
  row sep/.code={\tikzcd@sep{row}{#1}{}},
  column sep/.code={\tikzcd@sep{column}{#1}{}},
  bo row sep/.code={\tikzcd@sep{row}{#1}{between origins}},
  bo column sep/.code={\tikzcd@sep{column}{#1}{between origins}},
  bo column sep/.default=normal,
  bo row sep/.default=normal,
}
\def\tikzcd@sep#1#2#3{% re-defintion of original package macro!
  \pgfkeysifdefined{/tikz/commutative diagrams/#1 sep/#2}%
    {\pgfkeysalso{/tikz/#1 sep={\ifx\\#3\\1*\else1.7*\fi\pgfkeysvalueof{/tikz/commutative diagrams/#1 sep/#2},#3}}}%
    {\pgfkeysalso{/tikz/#1 sep={#2,#3}}}}
\makeatother

\title{Goodwillie's Calculus of Functors \\ and Higher Topos Theory}
\author{Mathieu Anel\footnote{SPHERE, UMR 7219, Univ. Paris Diderot. mathieu.anel@gmail.com} , \ \ 
Georg Biedermann\footnote{LAGA (UMR7539), Institut Galil\'ee, Univ. Paris 13; new address: Universidad del Norte, Departamento de Matem\'{a}ticas y Estad\'{i}stica, Barranquilla, Colombia, gbm@posteo.de} , \ \ 
Eric Finster\footnote{IRIF, Univ. Paris Diderot. ericfinster@gmail.com} , \\
and Andr\'{e} Joyal\footnote{CIRGET, UQ\`AM. joyal.andre@uqam.ca} 
}
\date{}

\begin{document}

\maketitle

\begin{abstract}
  We develop an approach to Goodwillie's Calculus of Functors using
  the techniques of higher topos theory.  Central to our method is the
  introduction of the notion of fiberwise orthogonality, a
  strengthening of ordinary orthogonality which allows us to give a
  number of useful characterizations of the class of $n$-excisive
  maps.  We use these results to show that the pushout product of a
  $P_n$-equivalence with a $P_m$-equivalence is a
  $P_{m+n+1}$-equivalence.  Then, building on our previous
  work~\cite{ABFJ:gbm}, we prove a Blakers-Massey type theorem for the
  Goodwillie tower of functors. We show how to use the resulting
  techniques to rederive some foundational theorems in the subject,
  such as delooping of homogeneous functors.
\end{abstract}

\setcounter{tocdepth}{2}
\tableofcontents

\section{Introduction}
\label{sec:introduction}

Goodwillie's calculus of homotopy functors \cite{G90, G92, G03} is a
powerful technique in homotopy theory for approximating possibly
complicated functors by simpler ones using a generalized notion of
excision.  In particular, applied to the identity functor on the
category of spaces, it produces a filtration interpolating between
stable and unstable homotopy which has proved extremely useful in
calculations.

In this article, we revisit some of the foundations of
the subject from the point of view of higher topos theory.  In
particular, we will show that many of the fundamental results can be
deduced from the following Blakers-Massey type theorem, which we feel
is of independent interest.
\vspace*{2mm}

\noindent
{\bf Theorem~\ref{thm:Goodwillie-Blakers-Massey}}
  Let
  \[
  \begin{tikzcd}
    F \ar[r, "g"] \ar[d, "f"'] & H \ar[d] \\
    G \ar[r] & K \pomark
  \end{tikzcd}
  \]
  be a homotopy pushout square of functors. If $f$ is a $P_m$-equivalence and $g$ is a $P_n$-equivalence, then the induced cartesian gap map
  \[ \gap{f}{g} : F \to G \times_K H \]
  \noindent is a $P_{m + n + 1}$-equivalence.
\vspace*{2mm}

The second result is a``dual'' version.
\vspace*{2mm}

\noindent
{\bf Theorem~\ref{thm:dual-Goodwillie-Blakers-Massey}}
Let
  \[
  \begin{tikzcd}
    F \pbmark \ar[r] \ar[d] & H \ar[d, "f"] \\
    G \ar[r, "g"] & K
  \end{tikzcd}
  \]
be a homotopy pullback square of functors. If $f$ is a $P_m$-equivalence and $g$ is a $P_n$-equivalence, then the cocartesian gap map
  \[ \cogap{f}{g}: G\sqcup_F H \to K\]
is a $P_{m + n + 1}$-equivalence.
\vspace*{2mm}

We show how to rederive known delooping results in homotopy
functor calculus in an easy and conceptual way as consequences. In
particular, we obtain a new proof of Goodwillie's Lemma 2.2 \cite{G03}
 that homogeneous functors deloop, independent of the material of
\cite[Section 2]{G03}

Both of these results rest on the material of the companion
article~\cite{ABFJ:gbm} where a very general version of the
Bla\-kers-Massey theorem was proved. There, the language of higher
topoi was adopted, and we find it equally well suited for the calculus
of homotopy functors, particularly because $n$-excisive functors to
spaces form a higher topos.  Indeed many of the results of this
article arise from working systematically fiberwise, a method very
much encouraged by the topos-theoretic point of view.  Given this
general framework, it will thus be convenient to drop any reference to
higher derived structures and take them for granted. When we talk
about a ``category'', we mean ``$\infty$-category'' and all
(co-)limits are to be interpreted as $\infty$-categorical
\mbox{(co-)}limits. In particular, we will not use the terms
``homotopy \mbox{(co-)}limit'', as was done above for the sake of
introduction. The $\infty$-categorical machinery already describes a
derived, homotopy invariant setting with all higher coherences.  We
will fequently say ``isomorphism'' for what is perhaps more commonly
called ``weak equivalence''. Similarly, mapping spaces or internal hom
objects are always to be taken derived.  The reader who finds this
article easier to read by using model structures should not have any
difficulties in doing so.

The main tool of our paper~\cite{ABFJ:gbm} was the notion of
\emph{modality}, and it will be equally important here. A modality is
a unique factorization system whose left and right class are closed
under base change.  An example is the factorization of a map of spaces
into an $n$-connected map followed by an $n$-truncated
map. Application of our generalized Blakers-Massey theorem to this
example leads to the classical version of the theorem. Here we observe
that the factorization of a natural transformation into a
$P_n$-equivalence followed by an $n$-excisive map is a modality in
a presheaf topos. That this is the case is a consequence of the fact
that Goodwillie's $n$-excisive approximation $P_n$ is, in fact, a left
exact localization of the topos of functors.  The left classes of
these $n$-excisive modalities for various $n\ge 0$ are compatible with
the pushout product in the following sense: \vspace*{2mm}

\noindent
{\bf Theorem~\ref{thm:adjunction-tricks}(4)}
The pushout product of a $P_m$-equivalence with a $P_n$-equi\-va\-lence is a $P_{m+n+1}$-equivalence.
\vspace*{2mm}

This fact immediately implies the main theorems by means of the
generalized Blakers-Massey theorems from~\cite{ABFJ:gbm}. It also
yields that the smash or join of an $m$-reduced functor with an
$n$-reduced functor is $(m+n-1)$-reduced, see
Example~\ref{ex:pointed-functor}.  In order to prove this theorem we
need to take a step back and develop systematically a fiberwise
approach, which is to say, concentrate our attention on constructions
which are compatible with base change. Indeed, one characterization of
a modality is a factorization system whose left and right classes
determine each other via \emph{fiberwise orthogonality}, a notion
which we introduce in Definition~\ref{defiberwiseorth0}.  Briefly, two
maps are fiberwise orthogonal if all of their base changes are
externally orthogonal to each other.  As is the case in the theory of
ordinary factorization systems, this condition can be reformulated as
say that a certain map, which we term the \emph{fiberwise diagonal}
(Definition~\ref{def:fiberwise-diagonal}) is an isomorphism.  The
resulting adjunction tricks exploited in
Proposition~\ref{prop:the-strange-adjunction} lead us to
Theorem~\ref{thm:n-excisive-equals-Pn-local} where we prove that the
$n$-excisive modalities are generated by pushout product powers of
certain explicit generating maps.

From here, Theorems~\ref{thm:Goodwillie-Blakers-Massey}
and~\ref{thm:dual-Goodwillie-Blakers-Massey} which provide analogues
of the Blakers-Massey theorems for the Goodwillie tower are easily
deduced. This allows us in Theorem~\ref{thm:delooping} to find a
classifying map for the map $P_nF\to P_{n-1}F$ in the Goodwillie tower
and reprove delooping results (Corollaries~\ref{cor:delooping}
and~\ref{thm:ADL}) for functors whose derivatives live only in a
certain range; in particular homogeneous functors are infinitely
deloopable.

To justify a portion of the result in~\ref{thm:delooping}, and because
we feel it is of independent interest,
Appendix~\ref{sec:appendix-epis-monos} is included. In
Theorem~\ref{thm:epi-mono-nexc} we give a characterization of
monomorphisms and covers (effective epimorphisms) in the topos
$\Snexc$ of $n$-excisive functors. As far as we know this is the first
time $n$-excisive functors are studied in detail as a topos and we
wish to advertise this as a fruitful line of thought.

Finally, a few remarks are in order about the overall placement of our
results in the larger landscape of studies of the Goodwillie Calculus.
Indeed, by now there are many versions of Goodwillie-style filtrations
which appear in a number of different contexts.  What is traditionally
known as the \emph{homotopy calculus} and concerns excisive properties
of functors (defined by their behavior on certain cubical diagrams)
can be developed in a very general setting as is done, for example,
model category theoretically in \cite{kuhn2007goodwillie} or
$\infty$-categorically in \cite{lurie2016higher}.  From the point of
view of this theory, the results of this article are somewhat restricted:
our arguments require that the functors under study have as codomain 
an $\infty$-topos.  In particular, this means we can not immediately
apply our results to functors taking values in stable categories, such
as, for example, spectra.  It remains for future work to understand to
what extent our techniques might be applied to the stable case.

From another point of view, however, our results can be seen as
providing a \emph{generalization} of the homotopy calculus.  For
example, when we restrict to functors with values in a topos, our
Theorem~\ref{thm:n-excisive-equals-Pn-local} leads to a completely
internal characterization of the construction of the Goodwillie tower
which makes no mention of cubical diagrams.  In upcoming work we will
show how these techniques can be applied to give a uniform treatment
of other varieties of Goodwillie calculus such as the \emph{orthogonal
  calculus} \cite{weiss1995orthogonal} where the approximation scheme
is not necessarily defined in terms of the behavior of a functor with
respect to limits and colimits.

This last point perhaps best illustrates the philosophy of our
approach.  Indeed, while the theory of higher topoi has applications
to Goodwillie calculus, providing streamlined and conceptual proofs of
its main results, the reverse is also true: we can use the Goodwillie
calculus as tool in the study of higher topoi themselves.  We feel
that much remains to be done in developing this point of view.

\medskip

{\bf Acknowledgment:} 
The first author has received funding from the European Research Council under the European Community's Seventh Framework Programme (FP7/2007-2013 Grant Agreement n$^{\circ}$263523). 
The second author and this project have received funding from the European Union’s Horizon 2020 research and innovation programme under Marie Sk\l odowska-Curie grant agreement No 661067.  
The second author also acknowledges support from the project ANR-16-CE40-0003 ChroK.
The third author has been supported by the CoqHoTT ERC Grant 64399.
The fourth author has been supported by the NSERCC grant 371436.

\section{Prerequisites}

In this section we recall material from the our companion
paper~\cite{ABFJ:gbm}.  In particular, we give the definition of a
modality~\ref{def:modality} and state our generalized Blakers-Massey
theorems~\ref{thm:gen-bm} and~\ref{thm:dual-gbm}.

\subsection{Topoi}
This article is written using the language of higher topoi.  For an
outline of the theory we refer the reader to~\cite{RezkTopos,
  JoyalNQC, LurieHT}.  A very brief overview of the essential
properties tailored to our needs is given in~\cite[Section
2]{ABFJ:gbm}.  We will now drop $\infty$ from the notation and refer
to them simply as topoi.  We write $\cS$ for the category of spaces.
We will denote the category of space-valued functors $\cC\to \cS$ on a
small category $\cC$ by $[\cC,\cS]$. A functor $\cC\to \cS$ is a
presheaf on $\cC\op$.

\begin{definition}
  A topos is an accessible left exact localization of a presheaf
  category $[\cC,\cS]$ for some small category $\cC$.
\end{definition}

The reader should be aware that ``left exact localization'' is to be
taken in the derived sense. Spelled out in the language of model
categories it means ``left Bousfield localization commuting with
finite homotopy limits up to weak equivalence''.  This is in line with
the general approach in this article that everything should be
interpreted in $\infty$-categorical terms. We recall that when we
speak of (co-)limits, the corresponding notions in the language of
model structures are homotopy (co-)limits.

\begin{example}
The primary examples of topoi of interest to us here are:
\begin{enumerate}
\item The category $\cS$ of spaces (as modelled by topological spaces
  or simplicial sets with weak homotopy equivalences) is the prime
  example of a topos.
\item The category $[\cC,\cS]$ of functors to spaces is a topos.
\item The full subcategory $[\cC, \cS]^{(n)} \subset [\cC, \cS]$ of
  \emph{$n$-excisive} functors, which, as explained in
  Example~\ref{exmp:Pn-yields-modality}, is itself a topos.
\end{enumerate}
\end{example}

We recall that within a topos colimits are preserved by base change.

\subsection{Cubes, gaps and cogaps}

Let $\ul{n}=\{1,\hdots,n\}$ and write $P(\ul{n})$ for the poset of its
subsets.  Define $P_0(\ul{n})$ to be the poset of non-empty subsets;
let $P_1(\ul{n})$ be the poset of proper subsets.

Now consider a finitely complete category $\cE$.  An \emph{$n$-cube}
in a category $\cE$ is a functor $\scr{X}: P(\ul{n}) \to \cE$. We will
refer to the canonical map
\[ 
\scr{X}(\emptyset)\to\lim_{U\in P_0(\ul{n})}\scr{X}(U)
\]
as the {\it cartesian gap map} or simply the {\it gap map} for brevity.
An $n$-cube is said to be \emph{cartesian} if its gap map is an isomorphism.
For example, a $2$-cube is cartesian if and only if it is a pullback square.

For an $n$-cube $\scr{X}$ in a finitely cocomplete category $\cE$
there also exists the canonical map
\[
\colim_{U \in P_1(\ul{n})} \scr{X}(U) \to \scr{X}(\ul{n}). 
\]
which we will call it the {\it cocartesian gap map} or briefly, {\it
  cogap map}.  An $n$-cube is \emph{cocartesian} if its cogap map is
an isomorphism.  A square is cocartesian if and only if it is a
pushout square.  

An $n$-cube is
called \emph{strongly cartesian} (resp. \emph{strongly cocartesian})
if all its $2$-dimensional subcubes are cartesian (resp. cocartesian). 

\begin{definition} \label{def:exterior-product-coproduct}
  \label{rem:iter-pp} The \emph{external cartesian product} 
  of two cubical diagrams $\scr{X}: P(\ul{m}) \to \cE$ and 
  $\scr{Y}: P(\ul{n}) \to \cE$ is a cubical diagram
  $\scr{X}\boxtimes \scr{Y}:P(\ul{m+n})=P(\ul{m})\times P(\ul{n}) \to \cE$
  defined by putting 
\[ 
(\scr{X}\boxtimes \scr{Y})(A,B)=\scr{X}(A)\times \scr{Y}(B)
\]
  for every $A\in P(\ul{m})$ and $B\in P(\ul{n})$.
  The \emph{external coproduct} $\scr{X}\boxplus \scr{Y}:P(\ul{m+n})=P(\ul{m})\times P(\ul{n}) \to \cE$
  is defined by putting 
\[
(\scr{X}\boxplus \scr{Y})(A,B)=\scr{X}(A)\sqcup \scr{Y}(B) .
\]
  \end{definition}

  The external cartesian (co)product of two strongly
  (co)cartesian cubes is strongly (co)cartesian.
  Every map $f:A\to B$ defines a 1-cube $f:P(\ul{1})\to \cE$.
  The $n$-cube 
\[
f_1\boxtimes \cdots \boxtimes f_n: P(\ul{n})\to \cE
\]
  is strongly cartesian and the 
  $n$-cube 
\[
f_1\boxplus \cdots \boxplus f_n: P(\ul{n})\to \cE
\]
  is strongly cocartesian for any
  sequence of maps $\{f_i : K_i \to L_i\}_{i=1}^{n}$. 
  In particular, the square 
\[
\begin{tikzcd}
       A\vee B \ar[r]\ar[d] & B \ar[d] \\ 
       A \ar[r] & \term
  \end{tikzcd}
\]
  is cocartesian for any pair of objects in a pointed 
  category.

\bigskip
Now let $\cE$ be finitely complete and finitely cocomplete.
Given a commutative square in $\cE$:
\[
  \begin{tikzcd}
       A \ar[r, "g"]\ar[d, "f"'] & C \ar[d, "k"] \\ B \ar[r, "h"] & D
  \end{tikzcd}
\]
We will denote the gap map by
\[
   \gap{f}{g}: A \to B\times_D C. 
\]
The cogap map of the square will be denoted by
\[   
   \cogap{h}{k}: B\cup_A C\to D .
\]
Strictly speaking these maps depend on the whole square. In practice
the remaining maps will always be clear from the context.

\subsection{Pushout product and pullback hom}
\label{subsec:pp-defn}

Let $\cE$ be a topos.
For any two maps $u: A\to B$ and $v: S\to T$ in $\cE$ the square
\[
  \begin{tikzcd}
    A\times S \ar[d, "u \times 1_S"'] \ar[r, "1_A \times v"] & A\times T \ar[d, "u \times 1_T"] \\
    B\times S \ar[r, "1_B \times v"'] & B\times T.
  \end{tikzcd}
\]
is cartesian, and we define the {\em pushout product of $u$ and $v$}, denoted $u\pp v$, to be the cocartesian gap map of the
previous square:
\[
  u\pp v = A\times T\sqcup_{A\times S} B\times S \to B\times T.
\]
Let $\ini$ and $\term$ be respectively an initial and a terminal
object for $\cE$. A topos $\cE$ has a {\it strict} initial object
which means that any arrow $C\to \ini$ is an isomorphism.  In
particular, since $\cE$ has finite products, this implies
$\ini\times X = \ini$ for all objects $X\in \cE$.  The pushout
product defines a symmetric monoidal structure on the category
$\cE^{\to}$ of arrows, with unit $\ini\to \term$. In particular, we
have $u\pp v= v\pp u$ and $(u\pp v)\pp w=u\pp(v\pp w)$.

\begin{example}\label{example-po}
  We give some examples of pushout products that will be useful in the sequel.
\begin{enumerate} 
  \item 
  The iterated pushout product $f_1\pp \cdots \pp f_n$
     of a sequence of maps  $\{f_i : K_i \to L_i\}_{i=1}^{n}$ 
     is the cogap map of $n$-cube $f_1\boxtimes \cdots \boxtimes f_n$.
  \item 
    For any map $A\to B$ in $\cE$ and any object $C$, the map
    $(\ini \to C)\pp (A\to B)$ is simply
\[
    (\ini \to C)\pp (A\to B) = C\times A\to C\times B
\]
  \item 
    For two pointed objects $\term\to A$ and $\term \to B$ in $\cE$, we
    have
\[
    (\term \to A)\pp (\term \to B) = A \vee B \to A \times B
\]
    the canonical inclusion of the wedge into the product.
  \item 

    Recall that the {\em join} of two objects $A$ and $B$ in $\cE$,
    denoted $A\join B$, is the pushout of the diagram
    $A\lto A\times B\to B$.  One sees immediately that
\[
    (A\to \term)\pp (B\to \term) = A\join B\to \term.
\]
\item The \emph{fiberwise join} $X\join_B Y$ of two maps $f:X\to B$ and
$g:Y\to B$ is the pushout  of the diagram $X\lto X\times_B Y\to Y$.
It is the pullback of the map $f\pp g$ along the diagonal $B\to B\times B$
\[
  \begin{tikzcd}
    X \sqcup_{X \times_B Y} Y \ar[r] \ar[d, "f \join_B g"'] \pbmark
    & (X \times B) \sqcup_{X \times Y} (B \times Y) \ar[d, "f \pp g"] \\
    B \ar[r] & B \times B.
  \end{tikzcd}
\]
The name fiberwise join is justified by the fact that for $b \in B$,
we have the identification
\[ \fib_b (X \join_B Y) = (\fib_b f) \join (\fib_b g) \]
since colimits are stable by base change in the topos.
More generally, the iterated fiberwise join $X_1\join_B \cdots \join_B X_n$
of a sequence of maps $f_i:X_i\to B$ ($1\leq i\leq n$) with
codomain $B$ is the the pullback of the map $f_1\pp \cdots \pp f_n$ 
along the diagonal $B\to B^n$.
  \item \label{fib-join}
    Since colimits in $\cE$ commute with base change, the pushout
    product $f\pp g$ can be thought as the ``external" join product of
    the fibers of $f$ and $g$. An easy computation shows that the
    fiber of the map $(f:A\to B)\pp (g:C\to D)$ at a point
    $(b,d)\in B\times D$ is the join of the fibers of $f$ and $g$.
    Details can be found in~\cite[Rem. 2.4]{ABFJ:gbm}.
  \item
    For an object $Z$ the slice category $\cE/Z$ has its own pushout
    product denoted $\pp_Z$. Given $f:A\to B$ and $g:X\to Y$ in $\cE/Z$,
    the corresponding formula reads
    \[
      f\pp_Z g = (A\times_Z Y)\cup_{(A\times_Z X)} (B\times_Z X) \to
      (B\times_Z Y).
    \]  
  \end{enumerate}
\end{example}

  We will make use of these observations in Section \ref{sec:goodwillie-loc}
  in order to relate the calculus of strongly cocartesian diagrams in
  a category $\cC$ with that of orthogonality in the presheaf
  category $[\cC, \cS]$.

\bigskip

For two objects $A$, $B$ of $\cE$, we let $\map{A}{B}$ be the space of
maps from $A$ to $B$ in $\cE$.  For two maps $u:A\to B$ and $f:X\to Y$ in
$\cE$ we consider the following commutative square in $\cS$
\[
  \begin{tikzcd}
    \map{B}{X} \ar[d] \ar[r] & \map{B}{Y} \ar[d] \\
    \map{A}{X} \ar[r]  & \map{A}{Y}.
  \end{tikzcd}
\]
We define the {\em external pullback hom} $\ph{u}{f}$ to be the cartesian gap map of the previous square:
\[
  \ph{u}{f} : \map{B}{X} \to \map{A}{X}\times_{\map{A}{Y}}\map{B}{Y}.
\]
Let $\intmap{A}{B}$ denote the internal hom object in $\cE$. Then we can define similarly an {\em internal pullback hom}
\[
  \intph{u}{f} : \intmap{B}{X} \to
  \intmap{A}{X}\times_{\intmap{A}{Y}}\intmap{B}{Y},
\]
which is a map in $\cE$.

\begin{example} \label{extrongorth}
The internal pullback hom has a number of useful special cases:
\begin{enumerate} 
  \item 
The category of spaces $\cS$ is cartesian closed
and we have $\intph{u}{f}=\ph{u}{f}$ for any pair of maps $u,f\in \cS$.
\item If $S^0 = \term\sqcup \term $ is the 0-sphere in $\cE$, then 
\[
\intph{S^0\to \term}{X\to \term} = X\to X\times X
\]
is the diagonal map of $X$.
Similarly, for any map $f:X\to Y$ the map
\[
\intph{S^0\to \term}{f} = X\to X\times_Y X
\]
is the diagonal map $\Delta f$ of $f$.

\item \label{ex:A-diagonal}
More generally, for any object $A$ in $\cE$, the map
\[
  \Delta_A(X)=\intph{A\to\term}{X\to\term}:X\to \intmap{A}{X}
\]
is the $A$-diagonal of $X$.
Similarly, for any map $f:X\to Y$ the map $\intph{A\to \term}{f}$ defines the $A$-diagonal of $f$.
\end{enumerate} 
\end{example}
 
It is useful to keep in mind that the \emph{global section functor} $\Gamma:=\map{\term}{-}:\cE \to \cS$ takes the object $\intmap{A}{B}$ to the space $\map{A}{B}$:
\[
\Gamma(\intmap{A}{B}) = \map{\term}{\intmap{A}{B}}=\map{A}{B}.
\]
Since $\Gamma$ commutes with all limits, one has
\[
\Gamma(\intph{f}{g})=\map{\term}{\intph{f}{g}} = \ph{f}{g}.
\]

For any $Z\in \cE$, the slice topos $\cE\slice{Z}$ has 
an internal hom and an internal pullback hom that we will denote respectively by $\intmap{-}{-}_Z$ and $\intph{-}{-}_Z$.
The base change $u^*:\cE\slice{Z}\to \cE\slice{T}$ along a map $u:T\to Z$, preserves cartesian products and internal homs.
For two objects $A$, $B$ in $\cE\slice{Z}$, we have a canonical isomorphism
\[
u^*\intmap{A}{B}_Z = \intmap{u^*A}{u^*B}_T.
\]
We leave to the reader the proof of the following lemma asserting that the same formula is true for the internal pullback hom.

\begin{lemma}\label{lem:base-change-intph}
For two maps $f:A\to B$ and $g:X\to Y$ in $\cE\slice{Z}$, and any map $u:T\to Z$ we have a canonical isomorphism in $\cE\slice{T}$:
\[
u^*\intph{f}{g}_Z = \intph{u^*f}{u^*g}_T.
\]
\end{lemma}

\begin{remark}
The external and internal pullback hom define functors
\[
  \ph{-}{-}:\left(\cE^{\to}\right)\op \times \cE^{\to} \to \cS^\to
\qquad\textrm{and}\qquad 
\intph{-}{-}:\left(\cE^{\to}\right)\op \times \cE^{\to} \to \cE^{\to}.
\]
Together with the pushout product the internal pullback hom yields a closed symmetric monoidal structure on $\cE^{\to}$. In particular, we have
\[ \intph{f\pp g}{h}=\intph{f}{\intph{g}{h}}. \]
We have also the relation $\ph{f\pp g}{h}=\ph{f}{\intph{g}{h}}$.
\end{remark}

\subsection{Orthogonality conditions}

In this section, we define and compare three notions of orthogonality
for maps: the {\em external orthogonality} $\perp$ and {\em internal
  orthogonality} $\iperp$, which will be related to the internal and
external pullback hom, and the new {\em fiberwise
  orthogonality~$\sperp$} (Definition \ref{defiberwiseorth0}), which
will be related to a variation of the pullback hom in Section
\ref{fwd}.

Although our focus is mainly on $\sperp$, it is convenient to
formulate its properties as properties of $\perp$. So we provide some
recollection on the matter. The relation $\iperp$ is introduced only
for comparison purposes and to avoid any confusion between $\sperp$
and $\intperp$.

Let us point out that our motivation for introducing fiberwise
orthogonality and the fiberwise diagonal is to prove
Proposition~\ref{prop:the-strange-adjunction}. This eventually leads
to Theorem~\ref{thm:adjunction-tricks} that is our new ingredient to
Goodwillie calculus that lets us prove the Blakers-Massey Theorem for
the Goodwillie tower.

\begin{definition}\label{def:ext-orth}
  Two maps $f:A\to B$ and $g:X\to Y$ in $\cE$ are {\em externally
    orthogonal} or simply {\em orthogonal} if the map $\ph{f}{g}$ is
  an isomorphism in $\cS$.  We write $f\perp g$ for this relation and
  we say that {\it $f$ is externally left orthogonal to $g$} and that
  {\it $g$ is externally right orthogonal to $f$}.
\end{definition}

Unfolding the definitions, one immediately verifies that if
$f \perp g$ then for every commutative square
\begin{equation*}
\begin{tikzcd}
    A \ar[d,"f"'] \ar[r, "h"] & X \ar[d, "g"] \\
    B \ar[r, "k"'] \ar[ru, dashed, "d"] & Y .
\end{tikzcd}
\end{equation*}
the space of diagonal fillers is contractible, that
is to say, a digonal filler exists and is unique up to homotopy.

Recall that for a topos $\cE$, all slice categories $\cE\slice{Z}$ are
also topoi.  Therefore, each $\cE\slice{Z}$ has an external orthogonality
relation which we will denote by $\relperp{Z}$.

\medskip

\begin{definition}\label{def:int-orth}
We will say that two maps $f:A\to B$ and $g:X\to Y$ are {\em internally orthogonal}, 
and write $ f \iperp g$, if the map $\intph{f}{g}$ is an isomorphism in $\cE$. 
Similarly we say that {\it $f$ is internally left orthogonal to $g$} and that {\it $g$ is internally right orthogonal to $f$}.
\end{definition}

For an object $Z$ and a map $f:A\to B$ we write $Z\times f: Z\times A\xrightarrow{\id_Z\times f} Z\times B$.
We leave to the reader the proof of the following lemma.

\begin{lemma}\label{lem:orthoint}
The following conditions are equivalent:
\begin{enumerate}
\item[\emph{(1)}] %\label{iperp}
	$f\iperp g$
\item[\emph{(2)}] %\label{times}
	For any $Z\in \cE$ we have $(Z\times f)\perp g$.
\end{enumerate}
In particular, $f\iperp g$ implies $f\perp g$.
\end{lemma}

Since each slice topos $\cE\slice{Z}$ has its own internal hom objects it has an internal orthogonality relation which we will denote by $\relintperp{Z}$.
The following lemma proves that these internal orthogonality relations are compatible with base change.

\begin{lemma}\label{lem:stabintortho}
For any two maps $f:A\to B$ and $g:X\to Y$ in $\cE$, and for any object $Z\in \cE$ we have
\[
f \iperp g \quad \Longrightarrow \quad Z^*f \reliperp{Z} Z^*g
\]
where $Z^*$ is the base change functor along the map $Z\to \term$.
Moreover, the converse is true if the map $Z\to \term$ is a cover.
\end{lemma}

\begin{proof}
If $u:Z\to \term$, then by Lemma \ref{lem:base-change-intph}, we have a canonical isomorphism
\[
u^*\intph{f}{g} = \intph{u^*f}{u^*g}_Z.
\]
This proves that $f\iperp g \Rightarrow u^*f\reliperp{Z} u^*g$.
The converse is true since the functor $u^*$ is conservative
when $u$ is a cover.
\end{proof}

The following proposition lists several equivalent properties that
will be used to define fiberwise orthogonality.  In order to
facilitate the reading, we employ the following convention in the
proofs which follow: given a map $f:A\to B$ and a map $u:Z\to B$, we
denote by $f_Z$ the map $u^*f:A\times_BZ\to Z$.  The point of this
notation is to make $u$ implicit, remembering only the new base.  The
context will make clear along which map the base is changed.

\begin{proposition}\label{prop:equivfwortho}
Given two maps $f:A\to B$ and $g:X\to Y$ in $\cE$, the following conditions are equivalent:
\begin{enumerate}
\item[\emph{(1)}] 
For any $Z\in \cE$ and any maps $b:Z\to B$ and $y:Z\to Y$, it is true in $\cE\slice{Z}$ that
\[
f_Z \reliperp{Z} g_Z.
\]
\item[\emph{(2)}] The base changes of $f$ and $g$ onto $B\times Y$ along the projections to $B$ and $Y$ satisfy 
\[
f_{B\times Y}  \reliperp{B\times Y} g_{B\times Y}.
\]
\item[\emph{(3)}] The diagonal map in $\cE\slice{B\times Y}$
\[
\Delta_{f_{B\times Y}}(g_{B\times Y}) : g_{B\times Y} \to \relintmap{B\times Y}{f_{B\times Y}}{g_{B\times Y}}
\] 
is an isomorphism (see \emph{Example \ref{extrongorth}.\ref{ex:A-diagonal}}).
\item[\emph{(4)}] 
For any $Z\to B\times Y$ and any $T\to Z$ we have 
\[ 
f_T\relperp{Z} g_Z. 
\]
\item[\emph{(5)}] For any $Z\in \cE$ and any maps $b:Z\to B$ and $y:Z\to Y$, it is true in $\cE\slice{Z}$ that
\[
f_Z  \relperp{Z}  g_Z.
\]
\item[\emph{(6)}] For any two maps $Z\to B$ and $Z'\to Y$ we have $f_Z\perp g_{Z'}$.
\item[\emph{(7)}] For any map $Z\to B$ we have $f_Z\perp g$.
\end{enumerate}
\end{proposition}

\begin{proof}
(1) $\Rightarrow$ (2) 
This is obvious since (2) is a special case of (1). 

(2) $\Rightarrow$ (1)
This follows from Lemma~\ref{lem:stabintortho} that states that orthogonality is stable by base change.

(2) $\Leftrightarrow$ (3) 
This is equivalent by the definition of orthogonality in $\cE_{B\times Y}$.

(1) $\Leftrightarrow$ (4)
This is Lemma \ref{lem:orthoint} applied to the topos $\cE\slice{Z}$.

(4) $\Rightarrow$ (5)
Set $T\to Z=\id_Z$.

(5) $\Leftrightarrow$ (6)
We need to prove that for all $Z$ and all $B\lto Z \to Y$,
\[
f_Z  \relperp{Z}  g_Z \iff \forall\, U\to B,\ \forall\, T\to Y,\ f_U \perp g_T .\]
We consider the following diagram
\[
\begin{tikzcd}
    U\times_BA \ar[d,"f_U"'] \ar[r] & U\times_YX \ar[d, "g_U"'] \ar[r] \pbmark & T\times_YX \ar[d, "g_T"] \\
    U \ar[r,equals] & U \ar[r,"h"] & T
\end{tikzcd}
\]
where $h$ is arbitrary and the right square is cartesian.
Because the right square is cartesian, the space of diagonal fillers of the outer square is equivalent to that of the left square. 
When $h$ varies, the former condition gives $f_U  \perp  g_T$ and the latter $f_U  \relperp{U}  g_U$, hence proving their equivalence.

(6) $\Leftrightarrow$ (7)
Since it is clear that (6) $\Rightarrow$ (7), we need to show the other implication.
Let $f_U$ be the base change of $f$ along some map $U\to B$, 
and $g_T$ the base change of $g$ along some map $h:T\to Y$, 
we consider the following diagram where the left square is commutative and the right square is cartesian
\[
\begin{tikzcd}
    U\times_BA \ar[d,"f_U"'] \ar[r] & T\times_YX \ar[d, "g_T"'] \ar[r] \pbmark & X \ar[d, "g"] \\
    U \ar[r] & T \ar[r,"h"] & Y.
\end{tikzcd}
\]
Again, because the right square is cartesian, the space of diagonal fillers of the outer square is equivalent to that of the left square, which proves (7) $\Rightarrow$ (6).

(7) $\Rightarrow$ (4)
Let us consider the following diagram
\[
\begin{tikzcd}
    T\times_BA \ar[d,"f_T"'] \ar[r, "k"] & Z\times_YX \ar[d, "g_Z"'] \ar[r] \pbmark & X \ar[d, "g"] \\
    T \ar[r] & Z \ar[r] & Y.
\end{tikzcd}
\]
where the right square is cartesian and $k$ is any map such that the left square is commutative. Condition~(4) says that for any such $k$ the space of fillers of the left square is contractible. Since the right square is cartesian this is equivalent to the outer square having a contractible space of fillers.
But Condition~(7) states that any map from $f_T$ to $g$, i.e. a commutative square, has a contractible space of fillers. So (7) implies (4).
\end{proof}

\begin{definition} \label{defiberwiseorth0} 
We will say that two maps $f:A\to B$ and $g:X\to Y$ are {\em fiberwise orthogonal} if they satisfy the equivalent properties of Proposition \ref{prop:equivfwortho}.
We will denote this relation by $f\sperp g$ and say that $f$ is {\em fiberwise left orthogonal} to $g$, and that $g$ is {\it fiberwise right orthogonal} to $f$.
\end{definition}

The intuitive idea behind this relation is that any fiber of $f$ is orthogonal to any fiber of $g$ in the external sense. 
This is the meaning of Condition~(6) where ``any fiber'' has to be understood as ``any pullback over an arbitrary base". 
Another way to understand fiberwise orthogonality is to say that it is the stabilization by base change of the relation $f\perp g$, which is the meaning of Condition~(5).

Condition~(7) helps to see that the relation $f\sperp g$ is stronger than the relation $f\iperp g$ since, by Lemma \ref{lem:orthoint},
the latter only requires that $Z\times f$ is orthogonal to $g$ for every object $Z\in \cE$.
Thus,
\[
f\sperp g \quad\Rightarrow\quad f\iperp g \quad\Rightarrow\quad f\perp g.
\]

\begin{remark}\label{remarksonfiberwiseorth}
We have the following immediate observations:
\begin{enumerate}
\item If $f$ is fiberwise left orthogonal to $g$, then every base
  change $f'$ of $f$ is left orthogonal to every base change $g'$ of
  $g$. Moreover, $f\sperp g \Rightarrow f'\sperp g' $.
\item The map $A\to \term$ for an object $A$ is fiberwise left
  orthogonal to a map $f:X\to Y$ if and only if it is internally left
  orthogonal to $f$. In particular two objects $A$ and $X$ are
  fiberwise orthogonal $(A\to \term)\sperp (X\to \term)$ if and only
  if they are internally orthogonal $(A\to \term)\iperp (X\to \term)$.
\end{enumerate}
\end{remark}

\subsection{The fiberwise diagonal map}\label{fwd}

We saw that the external and internal orthogonality of two maps $f$ and $g$ can be detected by the condition that some map ($\ph{f}{g}$ or $\intph{f}{g}$) be an isomorphism. The same thing is true for the fiberwise orthogonality, although the construction of the corresponding map is a bit more involved.

\begin{definition}\label{def:fiberwise-diagonal}
Take two maps $f: A\to B$ and $g:X\to Y$ in $\cE$; pull them back to the common target $B\times Y$, i.e. consider the maps
\[ 
  f_{B\times Y}= f\times\id_Y: A\times Y\to B\times Y
\]
and
\[ 
  g_{B\times Y}= \id_B\times g: B\times X\to B\times Y
\]
and view them as objects over $B\times Y$. In the slice $\cE\slice{B\times Y}$ one can form the $f_{B\times Y}$-diagonal of $g_{B\times Y}$ already used in~\ref{prop:equivfwortho}.(3).
We will denote this diagonal by $\fwd{f}{g}$ and name it the {\em fiberwise diagonal map}. 
\[
\fwd{f}{g}=\Delta_{f_{B\times Y}}(g_{B\times Y})=\intph{f_{B\times Y}}{g_{B\times Y}}_{B\times Y},
\]
where the internal pullback hom on the right is taken in the topos $\cE\slice{B\times Y}$. 
Explicitly, 
\[ 
\fwd{f}{g} :
\begin{tikzcd}
       B\times X \ar[d, "{(\id_B,g)}"'] \\
       B\times Y
    \end{tikzcd}  
\tto
\left\lsem \begin{tikzcd}
      A\times Y \ar[d, "{(f,\id_Y)}"] \\
      B\times Y  
    \end{tikzcd} ,
\begin{tikzcd}
       B\times X \ar[d, "{(\id_B,g)}"] \\
       B\times Y
    \end{tikzcd}\right\rsem_{B\times Y}  .
\]
\end{definition}

\begin{remark}
Let $b:\term\to B$ and $y:\term\to Y$ be points of $B$ and $Y$. We denote by $f_b$ and $g_y$ the corresponding fibers of $f$ and $g$. Since in a topos $\cE$ colimits commute with base change, the fiber of $\fwd{f}{g}$ at $(b,y)$ can be proven to be the diagonal map
\[
g_y\to \intmap{f_b}{g_y}.
\]
This is one of the reasons why we call this map the fiberwise diagonal map.
\end{remark}

\begin{proposition}\label{prop:fiberwisediagonalorth}
Let $f$ and $g$ be maps in $\cE$. Then $f\sperp g$ if and only if $\fwd{f}{g}$ is an isomorphism.
\end{proposition}

\begin{proof}
This is exactly the content of \ref{prop:equivfwortho}(3).
\end{proof}

We now arrive at our key technical result.
\begin{proposition}\label{prop:the-strange-adjunction}
The following formula is true in any topos:
\[
\fwd{f\pp g}{h} =  \fwd{f}{ \fwd{g}{h}}.
\]
\end{proposition}

For the proof of this proposition we need the following two auxiliary lemmas.

\begin{lemma}\label{lem:a-general-pullback}
For all $A$, $C$ and  $B \to C$ in any topos, the following square 
  \[
    \begin{tikzcd}
      \relintmap{C}{A \times C}{B} \ar[r] \ar[d] \pbmark & \intmap{A}{B} \ar[d] \\
      C \ar[r] & \intmap{A}{C},
    \end{tikzcd}
  \]
  where $\intmap{-}{-}_C$ is the internal hom in $\cE\slice{C}$
  and where the bottom map is the diagonal map,
is a pullback.
\end{lemma}

\begin{proof}
Using $C = \intmap{A\times C}{C}_C$ at the bottom left, we can factor the square as
  \[
    \begin{tikzcd}
      \relintmap{C}{A \times C}{B} \ar[r] \ar[d]  & \intmap{A}{B}\times C \ar[d] \ar[r] & \intmap{A}{B} \ar[d] \\
      \intmap{A\times C}{C}_C \ar[r] & \intmap{A}{C}\times C\ar[r] & \intmap{A}{C},
    \end{tikzcd}
  \]
Then, the right square is obviously cartesian.

To prove that the left square is also cartesian we use first the fact that the base change $\cE\to \cE\slice{C}$ preserves internal homs; this shows that $\intmap{A\times C}{B\times C}_C = \intmap{A}{B}\times C$.
Then the left square is cartesian as the image of the cartesian square in $\cE\slice{C}$
  \[
    \begin{tikzcd}
      B \ar[r] \ar[d]  & B\times C \ar[d] \\
      C \ar[r] & C\times C
    \end{tikzcd}
  \]
by the functor $\intmap{A\times C}{-}_C$ which preserves limits.
\end{proof}

\begin{lemma}\label{lem:a-specific-pullback}
The square
  \[
    \begin{tikzcd}
      \relintmap{\intmap{Y}{Z}}{X \times \intmap{Y}{Z} }{Z} \ar[r] \ar[d] \pbmark & \intmap{X}{Z}\ar[d] \\
      \intmap{Y}{Z} \ar[r] & \intmap{X \times Y}{Z}
    \end{tikzcd}
  \]
is a pullback. Hence, there is a canonical isomorphism
\[
  \intmap{X\join Y}{Z}=\relintmap{\intmap{Y}{Z} }{X \times \intmap{Y}{Z}}{Z}.
\]
\end{lemma}

\begin{proof}
  Setting $A = X$, $B = Z$ and $C = \intmap{Y}{Z}$ in the previous
  lemma we find that the square above is a pullback as claimed. Since
  the join is the pushout of the projections
\[ 
  X\leftarrow X\times Y\to Y,
\]
the pullback of this square is canonically isomorphic to $\intmap{X\join Y}{Z}$.
\end{proof}

\begin{proof}[Proof of Proposition~\ref{prop:the-strange-adjunction}]
We consider first the special case where the maps are of the following form
\begin{align*}
    f: X \to \term\ ,\ g: Y \to \term\ , \ h: Z \to \term.
\end{align*}
Then the map $\fwd{f \pp g}{h}$ becomes the $X\join Y$-diagonal of $Z$
  \[
  \fwd{ X \join Y \to \term}{Z \to \term} = Z \to \intmap{X\join Y}{Z}.
  \]
On the other hand, the map $\fwd{f}{\fwd{g}{h}}$ becomes
  \[ \fwd{ X  \to \term}{Z \to \intmap{Y}{Z} } =  Z \to \relintmap{\intmap{Y}{Z}}{X \times \intmap{Y}{Z}}{Z} . \]
Lemma~\ref{lem:a-specific-pullback} shows that these two maps are the same. 
This proves our claim in the special case.

We prove the general case by arguing fiberwise, i.e. by viewing our maps as objects in the respective slice categories and then appealing to the special case above.
We introduce the following convenient notation.
First, we will denote the cartesian product of two objects $I$ and $J$ in $\cE$ by concatenation $IJ$.
Then, for a map $f:X\to I$ in a topos $\cE$, we will abuse notation and denote by $X$ the corresponding object in $\cE\slice{I}$.
If another object $J\in \cE$ is given, we will denote by $X_J$ the base change of $X\in \cE\slice{I}$ to $\cE\slice{IJ}$ along the projection $I\times J\to I$, i.e. $X_J$ is the map $X\times J\to I\times J$.

For two maps $f:X\to I$ and $g:Y\to J$, the map $f \pp g$ in $\cE$ corresponds to the object 
\[
X_J \join Y_I
\]
in $\cE\slice{IJ}$, where the join is also computed in $\cE\slice{IJ}$. For a third object $K$, it is easy to compute that
\[
(X_J \join Y_I)_K = X_{JK} \join Y_{IK}
\]
in $\cE\slice{IJK}$.

Similarly, for two maps $g:Y\to J$ and $h:Z\to K$, the map $\fwd{g}{h}$ is defined as the map in $\cE\slice{JK}$
\[
\intph{Y_K\to \term}{Z_J\to \term} 
\]
where the pullback hom is computed in $\cE\slice{JK}$.
For a third object $I\in \cE$, because the pullback functor $\cE\slice{JK}\to \cE\slice{IJK}$ preserves exponentials, we have also 
\begin{align*}
   \left(\hspace{-5mm}\phantom{n^{n^n}}\intph{Y_K\to \term}{Z_J\to \term}\right)_I &= \left(\hspace{-6mm}\phantom{n^{n^n}}Z_J \to \intmap{Y_K}{Z_J}\right)_I = Z_{IJ} \to \intmap{Y_{IK}}{Z_{IJ}} \\
       & = \intph{Y_{IK}\to \term}{Z_{IJ}\to \term}
\end{align*}
in $\cE\slice{IJK}$.
Finally, we obtain the following canonical isomorphisms:
\[
\begin{array}{ll}
\rule[-1ex]{0pt}{4ex} \fwd{f\pp g}{h} & \textrm{viewed as a map in $\cE\slice{IJK}$} \\ 
\rule[-1ex]{0pt}{4ex} = \intph{(X_J \join Y_I)_K \to \term}{Z_{IJ}\to \term} & \textrm{join in $\cE\slice{IJ}$, bracket in $\cE\slice{IJK}$} \\ 
\rule[-1ex]{0pt}{4ex} = \intph{X_{JK} \join Y_{IK} \to \term}{Z_{IJ} \to \term} & \textrm{computed in $\cE\slice{IJK}$} \\ 
\rule[-1ex]{0pt}{4ex} = \intph{X_{JK}\to \term}{\intph{Y_{IK} \to \term}{Z_{IJ} \to \term}} & \textrm{special case applied to the topos $\cE\slice{IJK}$} \\ 
\rule[-1ex]{0pt}{4ex} = \intph{X_{JK}\to \term}{\left(\intph{Y_{K} \to \term}{Z_{J} \to \term}\right)_I} & \textrm{inside bracket computed in $\cE\slice{JK}$} \\ 
\rule[-1ex]{0pt}{4ex} = \fwd{f}{\fwd{g}{h}} & \textrm{viewed as a map in $\cE\slice{IJK}$}.
\end{array}
\]
\end{proof}

\subsection{Modalities and generalized Blakers-Massey theorems}
Given a class of maps $\cM$ of $\cE$, we write $\cM^{\perp}$ for the class of
maps which are externally right orthogonal to every map of $\cM$.  Similarly,
the class ${}^{\perp}\cM$ denotes the class of maps externally left orthogonal
to every map of $\cM$.

Recall that a {\em factorization system} on a category $\cE$ is the
data of a pair $(\cL,\cR)$ of classes of maps in $\cE$ such that
\begin{enumerate}
\item every map $f$ in $\cE$ can be factored in $f = rl$ where $l\in \cL$  and $r\in \cR$, and
\item $\cL^{\perp}=\cR$ and $\cL=\mbox{}^{\perp}\cR$.
\end{enumerate}
In a factorization system, the right class is always stable by base change.
\begin{definition}\label{def:modality}
  Let $\cE$ be a topos. A {\em modality} on $\cE$ is a factorization
  system $(\cL,\cR)$ such that the left class $\cL$ is also stable by
  base change.
\end{definition}

\begin{proposition}
  A factorization system $(\cL,\cR)$ is a modality if and only if the
  stronger orthogonality property $\cL\sperp \cR$ holds.
\end{proposition}

\begin{proof}
  The equivalence is given by Proposition \ref{prop:equivfwortho}.(7)
  which states exactly that the left class $\cL$ is stable by base
  change.
\end{proof}

An important source of modalities on a topos $\cE$ are provided by
the accessible left exact localizations of $\cE$. (These are, in fact,
exactly the \emph{subtopoi} of $\cE$, though we will not have occasion
to use this observation.)  To recall the construction, let
$F' : \cE \to \cE'$ be a functor with fully-faithful right adjoint
$i : \cE' \to \cE$.  As $i$ is fully-faithful, it is convenient to
work with the associated endofunctor $F = i \circ F'$, identifying
$\cE'$ with its corresponding reflective subcategory in $\cE$. We
now have the following standard definitions:

\begin{definition}
  \label{defn:lex-classes}
  Let $F : \cE \to \cE$ be as above.
  \begin{enumerate}
  \item A map $f : A \to B$ is said to be \emph{$F$-local} if
    the square
    \[
      \begin{tikzcd}
        A \ar[d, "f"'] \ar[r] & F(A) \ar[d, "F(f)"] \\
        B \ar[r] & F(B)
      \end{tikzcd}
    \]
    is cartesian.
  \item A map $f : A \to B$ is an $F$-equivalence if $F(f)$ is
    an isomorphism.
  \end{enumerate}
\end{definition}

\begin{lemma}\label{lem:lexloc-yields-modality}
  Let $F$ be a left exact localization of a topos $\cE$. If we let
  $\cL$ be the class of $F$-equivalences and $\cR$ the class of
  $F$-local maps, then $(\cL,\cR)$ forms a modality on $\cE$.
\end{lemma}

\begin{proof}
  Given a map $f : A \to B$, one may produce directly a factorization
  $f = v \circ u$ by first applying $F$ and defining $C$, $u$ and $v$
  by forming the pullback as in the following diagram
  \[
    \begin{tikzcd}
      A \ar[dr, "u"] \ar[ddr, "f"', bend right] & & \\
      & C \ar[r] \ar[d, "v"] \pbmark & F(A) \ar[d, "F(f)"] \\
      & B \ar[r] & F(B)
    \end{tikzcd}
  \]
  The map $v$ is $F$-local by construction, and one immediately checks
  that $u$ is an $F$-equivalence using the idempotence of $F$.  That
  the class of $F$-equivalences is stable by base change is clear
  from the fact that $F$ preserves finite limits.

  To check orthogonality we use the following observation.  Let
  $g\in\cR$ be any $F$-local map. Note that for any map $f$ we have
\[
    \ph{f}{g}=\ph{f}{Fg}=\ph{Ff}{Fg},
\]
  where the first equality comes from the fact that $g$ is a base
  change of $Fg$ and the second equality comes from the universal
  property of the localization $F$.  It follows that if $f$
  is an $F$-equivalence, so that $F(f)$ is an isomorphism, then
  $\ph{Ff}{Fg}=\ph{f}{g}$ is an isomorphism. This shows that
  $\cL \subseteq{}^\perp\cR$ and $\cL^\perp\supseteq\cR$.
  
  Now let $f: A \to B \in {}^\perp\cR$.  We must show that $f$ is
  an $F$-equivalence.  Consider the diagram
  \[
    \begin{tikzcd}
      A \ar[r, "u"] \ar[d, "f"'] & C \pbmark \ar[d, "v"] \ar[r] & F(A) \ar[d, "F(f)"] \\
      B \ar[r, equal] \ar[ur, dotted, "\exists ! h"] & B \ar[r] & F(B)
    \end{tikzcd}
  \]
  where $C$, $u$ and $v$ are defined by pullback.  Since $v$ is
  $F$-local by construction, we have $f \perp v$.  Hence we obtain a
  unique lift $h$.  One can easily check that $F(h)$ provides an
  inverse to $F(f)$ which shows that $F(f)$ is an equivalence.  This
  shows that $\cL = {}^\perp\cR$.

  Finally, let $g : X \to Y \in \cL^\perp$.  Factor $g$ as $g = v \circ u$
  as in the diagram
  \[
    \begin{tikzcd}
      X \ar[r, "u"] \ar[d, "g"'] & C \pbmark \ar[d, "v"] \ar[r] & F(X) \ar[d, "F(g)"] \\
      Y \ar[r, equal] & Y \ar[r] & F(Y)
    \end{tikzcd}
  \]
  so that $v$ is $F$-local and $u$ is an $F$-equivalence by
  construction.  We would like to show that the map $u$ is an
  isomorphism, as this implies that $g$ is $F$-local.  But now, we
  have a unique lift $h$ in the following diagram
  \[
    \begin{tikzcd}
      X \ar[d, "u"'] \ar[r, equal] & X \ar[d, "g"] \ar[d] \\
      C \ar[r, "v"'] \ar[ur, dotted, "\exists ! h"] & Y
    \end{tikzcd}
  \]
  since $u \perp g$ by assumption and one readily checks that $h$ is the
  required inverse.  This shows that $\cR = \cL^\perp$ and completes the
  proof.
\end{proof}

\begin{example}\label{exmp:Pn-yields-modality}
  Goodwillie's $n$-excisive approximation construction $P_n$ is a left
  exact localization of the $\infty$-topos $[\Cat{C},\Cat{S}]$ for
  some small category \Cat{C} with finite colimits and a terminal
  object. Hence, the $P_n$-equivalences and the $P_n$-local maps form
  a modality. This example is developed in detail
  Section~\ref{sec:goodwillie-loc}, see
  Definition~\ref{def:nexcisivemodality}.
\end{example}

Let $(\cL, \cR)$ on a topos $\cE$ and suppose we are give a
commutative square
\begin{equation}\label{eqn:the-square}
  \begin{tikzcd}
    Z \ar[d, "f"'] \ar[r, "g"] & Y \ar[d, "k"] \\
    X \ar[r, "h"] & W 
  \end{tikzcd}
\end{equation}

\begin{definition}
The square~\eqref{eqn:the-square}
is said to be {\it $\cL$-cartesian} if the gap map
\[
   \gap{f}{g} : Z \to X \times_W Y
\]
is in $\cL$.
The square is called {\it $\cL$-cocartesian} if the cogap map
\[
  \cogap{h}{k} : X\cup_Z Y\to W
\]
is in $\cL$.
\end{definition}

Given a map $f:X\to Y$, its {\it diagonal $\Delta f$} is the map
\[ 
    \Delta f : X\to X\times_YX 
\]
induced by pulling back $f$ along itself. In particular
$\Delta(X\to\term)$ is the classical diagonal $X\to X\times X$.

In \cite{ABFJ:gbm}, the following two facts were proven about this situation:
\begin{theorem}[Blakers-Massey {\cite[Thm. 4.0.1]{ABFJ:gbm}}]\label{thm:gen-bm}
Let Diagram~\eqref{eqn:the-square} be a pushout square.
Suppose that $\Delta f\pp \Delta g \in \cL$.  Then the square is $\cL$-cartesian.
\end{theorem}

\begin{theorem}[``Dual'' Blakers-Massey {\cite[Thm. 3.6.1]{ABFJ:gbm}}]\label{thm:dual-gbm}
Let Diagram~\eqref{eqn:the-square} be a pullback square.
Suppose that the map $h \pp k \in \cL$. 
Then the square is $\cL$-cocartesian.  
\end{theorem}

\section{The Goodwillie Localization}\label{sec:goodwillie-loc}
We will now revisit the Goodwillie n-excisive localization from the perspective
of topos theory. Our approach here is not the most general possible. In~\cite{BR13} a
reasonably general framework for Goodwillie calculus in the language of model
categories is developed. In~\cite{H15}, the author constructs Goodwillie approx-
imations of arbitrary categories. Here, however, we are particularly interested
in functor categories, and more specifically, those valued in spaces.

\subsection{The $n$-excisive modality}
All of our arguments can be carried out in the presheaf topos
$[\cC, \cS]$ where $\cC$ is a category with finite colimits and a
terminal object, and hence we will work in that level of generality.
We note that the standard examples of finite spaces ($\Fin$) and
finite pointed spaces ($\Fin_*$) fall into this category.  Moreover,
the class of such categories is closed under slicing and taking
pointed objects. It includes in particular the source categories used
by Goodwillie to construct the Goodwillie tower of a functor at a
fixed object.

We stress that the target category is {\it unpointed} spaces because
we rely on topos-theoretic arguments. No pointed category can be a non-trivial
topos. However, our main results~\ref{thm:adjunction-tricks},
\ref{thm:dual-Goodwillie-Blakers-Massey} and
\ref{thm:Goodwillie-Blakers-Massey} are still valid for functors with
values in pointed spaces. This follows from the observation that a
natural transformation of functors to pointed spaces is $n$-excisive
if and only it is still $n$-excisive after forgetting the basepoint,
and analogously for $P_n$-equivalences.

Let us fix in this section a category $\cC$ as above, writing $\term$
and $\ini$ for the terminal and initial objects respectively.  Recall
that the starting point for Goodwillie calculus is the following

\begin{definition}
  A functor $F : \cC \to \cS$ is $n$-excisive if it carries strongly
  cocartesian $(n+1)$-cubes in $\cC$ to cartesian cubes in $\cS$.
\end{definition}

In order to provide examples of $n$-excisive functors, Goodwillie
introduces the following construction. Given a functor $F : \cC \to \cS$,
define a new functor $T_nF$ by the formula:
\[ T_n F(K) := \lim_{U\in P_0(\ul{n+1})} F(K \join U) \] 
There is a natural map 
 $$t_nF : F \to T_n F$$ 
determined at an object $K$
by the cartesian gap map of the cube $U\mapsto F(K\join U)$.

\begin{remark}
  \label{rem:c-join}
  While we do not require that the category $\cC$ admits finite products,
  the above formula nonetheless makes sense in our setting.  Indeed, as
  $\cC$ admits finite coproducts, it admits a tensoring over the category
  of finite sets by setting
  \[ K \otimes U = \coprod_{U} K \]
  Since $\cC$ has a terminal object, we can regard $U$ as an object of
  $\cC$ by considering the object $\term \otimes U$.  One can easily check
  that this makes $K \otimes U$ into a product in $\cC$, so that one
  can define the join using the usual formula.  Equivalently, one may define
  $K \join U$ directly by the colimit:
\[
  K\join U=\colim \left \{
  \begin{tikzcd}
    & K\ar[ld]\ar[rd] \\
    \term &\dots & \term
  \end{tikzcd} \right \}
\]
  with $U$ copies of the terminal object appearing in the diagram.
  When $\cC$ is taken to be $\Fin$ or $\Fin_*$, this definition coincides
  with the standard one.
\end{remark}

With this construction in hand, we now iterate, defining a functor
$P_n F$ as the colimit of the induced sequence
\[ P_nF := \colim \{ F \to T_n F \to T^2_n F \to \cdots \} \] We
summarize the relevant facts about this construction with the
following %proposition.

\begin{proposition}[Goodwillie \cite{G03}]
  \label{prop:pn-properties}
  Let $F \in [\cC, \cS]$.
  \begin{enumerate}
  \item $P_nF$ is $n$-excisive.
  \item The functor $P_n : [\cC,\cS] \to [\cC,\cS]$ commutes with
    finite limits.
  \item The canonical map $F \to P_nF$ is universal for maps from $F$
    to $n$-excisive functors.  In particular, the functor $P_n$ is
    idempotent.
  \end{enumerate}
\end{proposition}

\begin{proof}
  The proofs appearing in \cite{G03}, as well as Rezk's streamlined
  version \cite{R08} are sufficiently general to go through in our
  setting with only minor modifications.  Indeed, for a translation of
  these arguments into the language of $\infty$-categories, the reader
  may consult \cite{lurie2016higher}[Section 6.1.1].
\end{proof}

Let us write $[\cC, \cS]\exc{n}$ for the full subcategory of
$n$-excisive functors.  The previous proposition can be summarized by
asserting that the functor
\[ P_n : [\cC, \cS] \to [\cC, \cS]\exc{n} \]
is a left exact localization \cite[Prop. 5.2.7.4]{LurieHT}.
In particular, $[\cC, \cS]\exc{n}$ is itself a topos \cite[Prop. 6.1.0.1]{LurieHT}). 

\begin{remark}\label{n-exc-colimits}
  The functor $P_n: [\cC,\cS] \to [\cC,\cS]$ commutes with filtered
  colimits. Since colimits in the localization $[\cC, \cS]\exc{n}$ are
  computed by reflecting the colimits of the ambient topos $[\cC,\cS]$
  down to $[\cC, \cS]\exc{n}$ via $P_n$, the functor $P_n$ viewed as
  taking values in $[\cC,\cS]\exc{n}$ actually commutes with all
  colimits:
\[
   P_n\colim_{i\in I} F_i= P_n\colim_{i\in I}P_n F_i.
\]
The right hand side is the colimit in $[\cC,\cS]\exc{n}$.
\end{remark}

As is the case for any left exact localization, the functor $P_n$
determines two classes of maps via Definition~\ref{defn:lex-classes}.
Moreover, according to Lemma~\ref{lem:lexloc-yields-modality}, these
two classes of maps form a modality.

\begin{definition}\label{def:nexcisivemodality}
We refer to the modality
\begin{center}($P_n$-equivalences, $P_n$-local maps)\end{center}
as the {\it $n$-excisive modality}.  
\end{definition}

Since the Generalized
Blakers-Massey theorem of \cite{ABFJ:gbm} applies to an \emph{arbitrary}
modality on a topos, we may apply the result already at this point, using nothing
but the left-exactness of the functor $P_n$.  The statement
obtained is the following:

\begin{proposition}
  Let
  \[
    \begin{tikzcd}
      F \ar[r, "g"] \ar[d, "f"'] & H \ar[d] \\
      G \ar[r] & K \pomark
    \end{tikzcd}
  \]
  be a pushout square of functors.  Suppose that $\Delta f \pp \Delta g$ is
  a $P_n$-equivalence.  Then so is the cartesian gap map
  \[ \gap{f}{g} : F \to G \times_K H \]
\end{proposition}

We think the reader will agree that the statement in its current form is
not entirely satisfactory: supposing that $f$ is a $P_i$-equivalence
and $g$ is a $P_j$-equivalence, we would like a determination of $n$ in
terms of $i$ and $j$. In the following sections, we will develop
the tools to make such a calculation using the calculus of orthogonality
developed above.  The final result is the following.

\begin{theorem}
  \label{thm:pp-prod-p-equiv}
  Let $f$ be a $P_i$-equivalence and $g$ a $P_j$-equivalence.  Then
  the map
  \[ \Delta f \pp \Delta g \]
  is a $P_{i+j+1}$-equivalence.
\end{theorem}

\subsection{Cubical Diagrams and Orthogonality}
\label{sec:cubic-diagr-orthg}

In order to prove Theorem \ref{thm:pp-prod-p-equiv}, we are going
to examine how the notion of cubical diagram in $\cC$ is transformed
by the Yoneda embedding
\[ y : \cC^{\op} \to [\cC, \cS]. \] 
We will see that there is a close
connection between strongly cocartesian cubical diagrams in $\cC$ and
fiberwise join products in $[\cC, \cS]$, leading to a number of useful
characterizations of the classes of $P_n$-equivalences and $P_n$-local
maps.  From here, the calculus of orthogonality, and in particular
the adjunction formula of Proposition~\ref{prop:the-strange-adjunction}
ultimately lead to the desired result.  In the discussion which
follows, we write 
\[
  R^K = \cC(K,-) = y(K) 
\]
for the representable functor
determined by an object $K \in \cC$. For a map $k:K\to L$ we write
\[
  R^k = y(k): \cC(L,-)\to \cC(K,-) 
\]
for the induced map.
We recall for later use
that the Yoneda embedding preserves all limits and hence
sends colimits in $\cC$ to limits in $[\cC, \cS]$.

Now let $\cX$ be a cubical diagram in $\cC$ and let us put
$K = \cX(\emptyset)$.  We denote by $\hat{\cX}$ the cubical diagram
obtained by composition with the (contravariant) Yoneda embedding.  That is,
$\hat{\cX} = y \circ \cX$.  The cocartesian gap map of this
cube takes the form
\[ \cogapcube{\cX} : \colim_{U \neq \emptyset} R^{\cX(U)} \to R^K \]
The interest in this map arises from the following elementary
observation:

\begin{lemma}
  \label{lem:cart-ortho}
  Let $F \in [\cC, \cS]$ be a functor.  Then 
  \[ \cogapcube{\cX} \perp (F\to\term) \iff \text{$F \circ \cX$ is cartesian} \]
\end{lemma}

\begin{proof}
  Unfolding the definition of the pullback hom $\ph{\cogapcube{\cX}}{F\to\term}$
  (and ignoring the trivial factors) we find 
  \[ \ph{\cogapcube{\cX}}{F\to\term}: [R^{\cX(\emptyset)}, F] \to [ \colim_U\, R^{\cX(U)}, F ]\]
  But of course $[R^{\cX(\emptyset)}, F] = F(\cX(\emptyset))$ and
  \begin{eqnarray*}
    [ \colim_U\, R^{\cX(U)} , F ]
    &=& \holim_U\, [ R^{\cX(U)}, F ] \\
    &=& \holim_U F(\cX(U))    
  \end{eqnarray*}
  by Yoneda.  Hence this is the map
  \[ \ph{\cogapcube{\cX}}{F\to\term}: F(\cX(\emptyset)) \to \holim_U F(\cX(U)) \]
  which is an isomorphism exactly if the cube $F \circ \cX$ is cartesian.
\end{proof}

\begin{cor}
  \label{cor:n-excisive-ortho}
  A functor $F \in [\cC, \cS]$ is $n$-excisive if and only if, for every
  strongly cocartesian $(n+1)$ cube $\cX$, we have $\cogapcube{\cX} \perp F\to\term $.
\end{cor}

In view of the previous corollary, it is natural to extend the
definition of $n$-excisiveness to maps so that a functor is
$n$-excisive if and only if the map $F \to \term$ is so. Concretely,
we have

\begin{definition}
  \label{defn:n-excisive-maps}
  A map $f : F \to G$ of functors is said to be {\it $n$-excisive} if
  for all strongly cocartesian $(n+1)$-cubes $\cX$ we have $\cogapcube{\cX} \perp f$.
\end{definition}

For convenience we note that $f$ is $n$-excisive if and only if for all strongly cocartesian $\cX$ as above, the square
  \[
    \begin{tikzcd}
      F(\cX(\emptyset)) \ar[r] \ar[d] & \lim_{U \neq \emptyset} F(\cX(U)) \ar[d] \\
      G(\cX(\emptyset)) \ar[r] & \lim_{U \neq \emptyset} G(\cX(U)) 
    \end{tikzcd}
  \]
is a pullback.

The following construction is a useful source of strongly cocartesian
diagrams.  The reader may wish to compare \cite[Example
2.8]{bauer2015cross} where a similar construction is considered.

\begin{construction}
  \label{cons:scc}
  Let $\{ k_i : K_i \to L_i \}_{i=1}^{n}$ be a family of maps in $\cC$.
  For $U \subseteq \ul{n}$, define
  \[
    \sigma_U(k_i) = \begin{cases}K_i & i \notin U \\ L_i & i \in U\end{cases}
  \]
  Associated to the family $\{k_i\}$ is a $n$-cubical diagram $\cK$
  defined by the formula
  \[ \cK(U) = \bigsqcup_{1 \leq i \leq n} \sigma_U(k_i) \]
  where for $U \subseteq V$, the induced map $\cK(U) \to \cK(V)$ is given by
  \[ \cK(U \hookrightarrow V) = \begin{cases} k_i & i \in V \setminus U \\ \id_{\sigma_V(i)} & \text{otherwise} \end{cases} \]
\end{construction}

\begin{lemma}
  \label{lem:gen-cross-effect-cube-scc}
  For any family of maps $\{ k_i : K_i \to L_i \}_{i = 1}^{n}$
  the cubical diagram $\cK$ is strongly cocartesian.
\end{lemma}

\begin{proof} 
  In the notation of Definition 2.2.1, we have $\cK=k_1\boxplus \cdots \boxplus k_n$. 
 \end{proof}

\begin{lemma}
  \label{lem:gen-cross-effect-cube-is-box}
  \[ \cogapcube{\cK} = R^{k_1} \pp \cdots \pp R^{k_n} \]
\end{lemma}

\begin{proof}
We have $\hat{\cK}:= y \circ \hat{\cK}=
 R^{k_1}\boxtimes \cdots \boxtimes R^{k_n}$,
 since the Yoneda functor takes coproduct to product.
 Hence the cocartesian gap map of $\hat{\cK}$
 is equal to $R^{k_1} \pp \cdots \pp R^{k_n}$.
\end{proof}

\begin{example}
  \label{exmp:cross-effect}
  Suppose the category $\cC$ is \emph{pointed}, that is, that the
  initial and terminal objects coincide in $\cC$.  It will be
  convenient in this case to write $\vee$ for the coproduct in $\cC$
  in order to make contact with the traditional notation.  In
  particular, we have $K \vee \term = K$ for all objects
  $K \in \cC$.

  Now consider a family of objects $\{K_i\}_{i = 1}^{n}$ in $\cC$.
  Applying Construction~\ref{cons:scc} to the collection of maps
  $\{ K_i \to \term \}_{i = 1}^{n}$ we find that the resulting cube may
  be more simply described as 
  \[ \cK(U) = \bigvee_{i \notin U} K_i \]
  Now let $F : \cC \to \cS$ be a functor.  Unraveling the definition
  shows that the pullback hom $\ph{\cogapcube{\cK}}{F}$ is the map
  \[ F(\bigvee_{i} K_i) \to \lim_{U \neq \emptyset} F(\bigvee_{i\notin U} K_i) \]
  The fiber of this map is what Goodwillie refers to as the
  \emph{$n$-th cross-effect}, writing $(cr_n F)(K_1, \dots, K_n)$.  It
  follows immediately from these considerations that we have
  $\cogapcube{\cK} \perp (F\to\term)$ for every family $\{K_i\}_{i = 1}^{n}$ of
  objects of $\cC$ if and only if $F$ is \emph{of degree $(n-1)$} in
  the sense of \cite[Definition 3.21]{bauer2015cross}; that is, $cr_n F$ vanishes.

  It is well known that to be of degree $n$ is strictly weaker than
  to be $n$-excisive.  Nonetheless, we will show below that we can
  recover the notion of $n$-excisiveness from the cubical diagrams
  $\cK$ by replacing the external orthogonality relation $\perp$ by
  the stronger fiberwise orthogonality relation $\sperp$.  It is
  exactly this observation which motivates the introduction of this
  stronger notion.
\end{example}

Construction \ref{cons:scc} turns out to be quite general: in fact, as
we now show, \emph{every} cubical diagram can be obtained from it
after a single cobase change.  To make this precise, suppose we are
given a strongly cocartesian cube $\cX : P(\ul{n}) \to \cC$.  Let us
put $K = P(\emptyset)$ and $K_i = P(\{i\})$.  The functorial action of
$\cX$ gives us maps
\[ k_i := \cX(\emptyset \to \{i\}) : K \to K_i . \]
Applying Construction \ref{cons:scc}, we obtain a new cubical diagram
which, in this case, we will denote by $\cX_\pp$ (the notation being
inspired by Lemma \ref{lem:gen-cross-effect-cube-is-box} above).
Unwinding the definition, we find that
\[ \cX_{\pp}(\emptyset) = K^{\sqcup n} \] so that the codiagonal
$\nabla : K^{\sqcup n} \to K$ provides a canonical map
$\cX_{\pp}(\emptyset) \to \cX(\emptyset)$.

\begin{lemma}
  \label{lem:scc-cobase}
  The strongly cocartesian cube $\cX$ is obtained from $\cX_{\pp}$
  by cobase change along the codiagonal map
  \[ \nabla : K^{\sqcup n} \to K \]
\end{lemma}

\begin{proof}
  The lemma asserts that for any $U \subseteq \ul{n}$, the square
  \[
    \begin{tikzcd}
      K \sqcup \cdots \sqcup K \ar[r] \ar[d] & K \ar[d] \\
      \bigsqcup_{1 \leq i \leq n} \sigma_U(i) \ar[r] & \cX(U)
    \end{tikzcd}
  \]

  \noindent is a pushout.  But since $\cX$ is strongly cocartesian,
  we have that

  \[ \cX(U) = \colim \left \{ \begin{tikzcd} & K_{i_1} \ar[d, phantom,
        "\vdots"] \\ K \ar[ur] \ar[dr] \ar[r] & K_{i_2} \ar[d, phantom,
        "\vdots"] \\ & K_{i_3} \end{tikzcd} \right \}_{i_k \in U} \]

  \noindent and one easily sees that this coincides with the pushout
  above by a simple cofinality argument.
\end{proof}

It is immediate from the previous lemma and the fact that the Yoneda
embedding sends colimits in $\cC$ to limits in $[\cC, \cS]$ that the
corresponding cube of representable functors $\hat{\cX}$ is obtained from
$\hat{\cX_{\pp}}$ by base change along the $n$-fold diagonal map
\[ \Delta : R^{\cX(\emptyset)} \to (R^{\cX(\emptyset)})^{\times n}
  = R^{\cX_{\pp}(\emptyset)} \] 
It will be convenient in what 
follows to introduce special notation for the cocartesian gap maps of
these two cubes.  We will use this notation exclusively in the case
where the given cubical diagram $\cX$ is known to be strongly
cocartesian. In this case, the cocartesian gap map of the cube $\hat{\cX}$
will be denoted
\[ \gamma^{\cX} : \Gamma^{\cX} \to R^{\cX(\emptyset)} \] 
where $\Gamma^{\cX} := \colim_{U \neq \emptyset} R^{\cX(U)}$. For $\hat{\cX_{\pp}}$, on the other hand, we will write
\[ w^{\cX} : W^{\cX} \to (R^{\cX(\emptyset)})^{\times n} \] 
with $W^{\cX}$ defined by the analogous colimit for the cube
$\cX_{\pp}$. Some justification for this special notation will be
given in Remark \ref{rem:ganea} below.  For now we observe

\begin{lemma}
  \label{lem:cocart-pb}
  For any strongly cocartesian cube $\cX$ in $\cC$, the square
  \[
    \begin{tikzcd}
      \Gamma^{\cX} \ar[d, "\gamma^{\cX}"'] \ar[r] \pbmark & W^{\cX} \ar[d, "w^{\cX}"] \\
      R^{\cX(\emptyset)} \ar[r, "\Delta"'] & (R^{\cX(\emptyset)})^{\times n}
    \end{tikzcd}
  \]
is a pullback in $[\cC, \cS]$.
\end{lemma}

\begin{proof}
  Immediate since colimits in $[\cC, \cS]$ are stable by base change.
\end{proof}

Combining Lemma~\ref{lem:gen-cross-effect-cube-is-box} with
the definition of the fiberwise join, we deduce immediately that

\begin{lemma}
  \label{lem:scc-is-fw-join}
  For any strongly cocartesian cubical diagram $\cX$ in $\cC$, the
  cocartesian gap map $\gamma^{\cX}$ of the cube of representable
  functors $\hat{\cX}$ is given by the expression
  \[ \gamma^{\cX} = R^{k_1} \join_{R^K} \cdots \join_{R^K} R^{k_n} \]
  where $k_i = \cX(\emptyset \hookrightarrow \{i\}) : K \to K_i$.
\end{lemma}

The above discussion has an important special case, which we now
describe.  Note that a given
strongly cocartesian diagram $\cX$ is completely determined by the
family of maps
\[\left\{ f_i : \cX(\emptyset) \to \cX(\{i\}) \right\}_{1 \leq i \leq n}\]
Consequently, we may identify the category of strongly cocartesian
$n$-cubes $\cX$ such that $\cX(\emptyset) = K$ with the $n$-th
cartesian power of the coslice category $(\Cat{C}_{K/})^{\times n}$.
As $\cC$ has a terminal object, this category clearly has one
as well, an $n$-cube which we will denote by $\cT^K_n$ and which is
determined by $\cT^K_n(\emptyset) = K$ and $\cT^K_n(\{i\}) = \term$
for $1 \leq i \leq n$. More generally, the reader can easily check
that we have
\[ \cT^K_n(U) = K \join U \] in the sense of Remark \ref{rem:c-join}.
Applying Yoneda as in the proof of \ref{lem:cart-ortho} we find that
\[ \ph{\gamma^{\cT^K_n}}{F}= (t_nF)(K): F(K) \to \holim_U F(K \join U) \]
is Goodwillie's map $t_nF$ introduced
in the previous section.  As these distinguished strongly cocartesian
cubes play a central role in the theory and are entirely determined
by the given object $K$, it will be convenient to use the abbreviation
\[ \gamma^K_n : \Gamma^K_n \to R^K \]
and
\[ w^K_n : W^K_n \to (R^K)^{\times n}\] for the maps $\gamma^{\cT^K_n}$
and $w^{\cT^K_n}$ constructed above. Note that by construction
  $$ w^K_n=(w^K_1)^{\square\, n} $$
where $w^K_1=R^{K\to\term}:R^{\term}\to R^K$.
In this case, then the statement
of Lemma \ref{lem:cocart-pb} asserts that the square
\begin{equation}\label{eqn:fundamentalpullback}
  \begin{tikzcd}
    \Gamma^{K}_n \ar[d, "\gamma^{K}_n"'] \ar[r] \pbmark & W^{K}_n \ar[d, "w^{K}_n"] \\
    R^{K} \ar[r, "\Delta"'] & (R^{K})^{\times n}
  \end{tikzcd}
\end{equation}
is a pullback for any $K \in \cC$.

\begin{remark}
  \label{rem:ganea}
The pullback diagram~\eqref{eqn:fundamentalpullback} is analogous to a well-known construction in
  classical homotopy theory.  For a pointed space $(X,x)$,
  the $n$-fold \emph{fat wedge} of $X$, denoted $W_n(X)$ may be
  defined as the iterated pushout product
  \[
    \begin{tikzcd}
      W_n(X) \ar[d, "w_n"'] \\
      X^{\times n}
    \end{tikzcd} := \hspace{1ex}
    \left (
    \begin{tikzcd}
      \term \ar[d, "x"'] \\ X
    \end{tikzcd}
    \right )^{\pp n} .
  \]
  Note that it comes equipped with a canonical inclusion $w_n$ into the
  $n$-fold product as shown.  The pullback of this map along the
  diagonal $X \to X^{\times n}$ is known as the $n$-th \emph{Ganea fibration},
  and denoted $\Gamma_n(X)$.
  \[
    \begin{tikzcd}
      \Gamma_n(X) \ar[d, "\gamma_n"'] \ar[r] & W_n(X) \ar[d, "w_n"] \\
      X \ar[r, "\Delta"'] & X^{\times n}
    \end{tikzcd}
  \]
  Recall from Subsection~\ref{subsec:pp-defn} that the pullback of an
  $n$-fold pushout product along the diagonal map is called the $n$-fold
  fiberwise join.  Thus the map $\gamma_n$ may alternatively be described as
  \[ \gamma_n = \term \join_X \cdots \join_X \term \]
  From the discussion of the fiberwise join, then, it is immediately clear
  that the fiber of the map $\gamma_n$ has the description
  \[ \fib_x \gamma_n = (\Omega X)^{\join n} \]
  as is well known.

  In fact, this construction makes sense in any topos.  Returning to the
  situation at hand, when the category $\cC$ is pointed, we find that
  the representable functor $R^{\term}$ is in fact the terminal functor
  in $[\cC, \cS]$.  Hence for any object $K \in \cC$, the terminal map
  $K \to \term$ provides the representable functor $R^K$ with a canonical
  base point
  \[ R^{\term} \to R^K \] 
Examining the pullback diagram~\eqref{eqn:fundamentalpullback} above,
  we find that it is exactly the $n$-th Ganea fibration of the
  representable $R^K$ as calculated in the topos $[\cC, \cS]$, which
  is the justification for the notation introduced above.  From this
  perspective, Theorem \ref{thm:n-excisive-equals-Pn-local} (4) below
  may be read as saying that the Goodwillie localization of the
  functor category $[\cC, \cS]$ is obtained by inverting the $n$-th
  Ganea fibration of the representable $R^K$ for all $K \in \cC$.  

Let us also point out that diagram~\eqref{eqn:fundamentalpullback} is well-defined and still a pullback even if $\cC$ is not pointed.
\end{remark}

\subsection{A Characterization of $n$-excisive maps}
\label{subsec:n-excisive-maps}

We now proceed to give a number of characterizations of the class of
$n$-excisive maps as defined above.  The reader will perhaps not be
surprised to learn that they coincide with the $P_n$-local maps
determined by the localization functor
$P_n : [\cC, \cS] \to [\cC, \cS]$, though this is not a priori
obvious.  Furthermore, characterization (2) in the following theorem
provides the main tool for establishing the compatibility of
$P_n$-equivalences with the pushout product.

\begin{theorem}\label{thm:n-excisive-equals-Pn-local}
Let $f:F\to G$ be a map in $[\cC, \cS]$. The following statements are equivalent:
\begin{enumerate}
\item[\emph{(1)}] For every family of maps $\{h_i : K_i \to L_i\}_{i = 0}^{n}$ in $\cC$ we have
  \[R^{h_0} \pp \cdots \pp R^{h_n} \sperp f\]
\item[\emph{(2)}] For all $K \in \cC$ we have $w^{K}_{n+1}\sperp f$.
\item[\emph{(3)}] For every family of maps $\{h_i : K \to K_i \}_{i = 0}^{n}$ in $\cC$ we have
  \[R^{K_0} \join_{R^K} \cdots \join_{R^K} R^{K_n} \perp f \]
\item[\emph{(4)}] For all $K \in \cC$ we have $\gamma_{n+1}^K\perp f$.
\item[\emph{(5)}] The map $f$ is $P_n$-local. 
\item[\emph{(6)}] The map $f$ is $n$-excisive.
\end{enumerate}
\end{theorem}

\begin{proof}
  We will begin with the equivalences $(3) \Leftrightarrow (4) \Leftrightarrow (5) \Leftrightarrow (6)$.

  $(3) \Rightarrow (4)$ This is the special case $K_i = \term$ for all $i$

  $(4)\Rightarrow (5)$ Examining the definition, we find that
  $\ph{\gamma^{K}_{n+1}}{F}$ is the cartesian gap map of the
  commutative square
  \[
    \begin{tikzcd}
      F(K) \ar[r]\ar[d] & \lim\limits_{U\neq\emptyset}F(K\term U) \ar[d] \\ 
      G(K) \ar[r] & \lim\limits_{U\neq\emptyset}G(K\term U),
    \end{tikzcd}
  \]
  Hence if $\gamma_{n+1}^K\perp f$, this square is a pullback.
  Recognizing the right vertical map as $T_n(f)$, it follows that $f$
  is a pullback of $T_n(f)$.  But then it is a pullback of all
  composites $T_n^kf$ because $T_n$ preserves finite limits. Since
  finite limits commute with filtered colimits in \Cat{S}, $f$ is a
  pullback of $P_nf=\colim_kT_n^kf$, ie.\! $f$ is $P_n$-local.
  
  $(5) \Rightarrow (6)$ Now assume that $f$ is $P_n$-local and let $\cX$ be a strongly cocartesian $(n+1)$-cube. Write $K=\cX(\emptyset)$. Consider the following commutative diagram:
  \[
  \begin{tikzcd}[bo column sep=large]
    & P_nF(K) \ar[rr] \ar[dd] & & \lim_{U\neq\emptyset} P_nF(\scr{X}(U)) \ar[dd] \\
    F(K) \ar[rr, crossing over] \ar[dd] \ar[ur] & & \lim_{U\neq\emptyset} F(\scr{X}(U)) \ar[ur] & \\
    & P_nG(K) \ar[rr] & & \lim_{U\neq\emptyset} P_nG(\scr{X}(U)) \\
    G(K) \ar[rr] \ar[ur] & & \lim_{U\neq\emptyset} G(\scr{X}(U)) \ar[ur]\ar[from=uu, crossing over] & 
  \end{tikzcd}
  \]
  We need to show that the front face is a pullback (see Definition
  \ref{defn:n-excisive-maps}) . The right and left faces are a
  pullbacks by assumption. The back square is trivially a pullback:
  both horizontal maps are isomorphisms because $P_nF$ and $P_nG$ are
  $n$-excisive functors by Proposition~\ref{prop:pn-properties}. 
 Thus, the composite diagonal square is a
  pullback. Hence, the front is also a pullback.

  $(6) \Rightarrow (3)$ According to Lemma~\ref{lem:scc-is-fw-join}, the cocartesian
  gap map of any strongly cocartesian diagram can be expressed in this form.
  Hence if $f$ is $n$-excisive, it is orthogonal to such a map by definition.
  
  We now treat statements $(1)$ and $(2)$.
  
  $(1) \Rightarrow (2)$ This is the special case where $h_i = K \to \term$ for all $i$.

  $(2) \Rightarrow (3)$ We have seen above that there is a pullback square
  \[
    \begin{tikzcd}
      \Gamma^{K}_{n+1} \ar[d, "\gamma^{K}_{n+1}"'] \ar[r] \pbmark & W^{K}_{n+1} \ar[d, "w^{K}_{n+1}"] \\
      R^{K} \ar[r, "\Delta"'] & (R^{K})^{\times n+1}
    \end{tikzcd}
  \]
  for any $K \in \cC$. But by the Definition~\ref{defiberwiseorth0} of fiberwise orthogonality, or more precisely by Proposition~\ref{prop:equivfwortho}(7), $f$ is orthogonal to any pullback of the map $w^{K}_{n+1}$, in particular $\gamma^{K}_{n+1}$ as claimed.
  
  $(6) \Rightarrow (1)$ Lemma  \ref{lem:gen-cross-effect-cube-is-box} identifies the map
  \[ R^{h_0} \pp \cdots \pp R^{h_n} \] as the cocartesian gap map of
  the strongly cocartesian cube determined by Construction
  \ref{cons:scc} and so the relation
  $R^{h_0} \pp \cdots \pp R^{h_n} \perp f$ holds by definition.  It
  remains to show that $f$ is orthogonal to any base change of this
  map.  But since we already know $(6) \Rightarrow (5)$, $f$ is
  $P_n$-local.  The result now follows since $P_n$-equivalences
  are stable by base change.
\end{proof}

\begin{remark}\label{rem:degree-n-warning}
  It is not possible to replace the fiberwise orthogonality relation
  $\sperp$ it items (1) and (2) with the weaker external orthogonality
  relation $\perp$.  Indeed, as pointed out in Example
  \ref{exmp:cross-effect}, the latter notion detects functors which
  are \emph{of degree $n$}, a strictly weaker condition.  
\end{remark}

\begin{remark}
In~\cite{ABFJ:gbm} we deduce the classical Blakers-Massey theorem from our generalized version by using the fact that the $n$-connected/$n$-truncated modalities are generated by pushout product powers of the map $S^0\to\term$. Theorem~\ref{thm:n-excisive-equals-Pn-local}(2) states that in the same sense the Goodwillie tower, that is the $n$-excisive modalities, are generated by the pushout product powers of the maps $w^K_1: R^\term\to R^K$ for all $K$ in $\cC$.
\end{remark}

We can now prove the main result of this section.  
Recall the fiberwise diagonal $\fwd{f}{g}$ of two maps $f$ and $g$ as defined
in~\ref{def:fiberwise-diagonal}. By Proposition~\ref{prop:fiberwisediagonalorth} it is an isomorphism if and only if the maps $f$ and $g$ are fiberwise
orthogonal. A crucial role in the proof of the next theorem is
played by the formula $\fwd{f\pp g}{h} = \fwd{f}{\fwd{g}{h}}$
demonstrated in Proposition~\ref{prop:the-strange-adjunction}. It
allows us to use adjunction tricks for fiberwise orthogonality.  The
reader is invited to compare the next theorem
with~\cite[Cor. 3.15]{ABFJ:gbm} where the $n$-connected/$n$-truncated
modalities for $n\ge -2$ are treated. 

\begin{theorem}
  \label{thm:adjunction-tricks}  
  Let $f$ be a $P_m$-equivalence, $g$ a $P_{n}$-equivalence and $h$ a
  $p$-excisive map.  Then:
  \begin{enumerate}
  \item[\emph{(1)}] The map $\fwd{w^K_n}{h}$ is $(p-n)$-excisive
  \item[\emph{(2)}] The map $f \pp w^K_n$ is a $P_{n+m}$-equivalence
  \item[\emph{(3)}] The map $\fwd{f}{h}$ is $(p-m-1)$-excisive
  \item[\emph{(4)}] The map $f \pp g$ is a $P_{m+n+1}$-equivalence
  \end{enumerate}
\end{theorem}

\begin{proof}
  (1). It is immediate from Lemma~\ref{lem:gen-cross-effect-cube-is-box} that
  \[ w^K_{p+1} = w^K_{p-n+1} \pp w^K_n \]
  Therefore, by Proposition~\ref{prop:the-strange-adjunction} we have
  \[ \fwd{w^K_{p+1}}{h} = \fwd{w^K_{p-n+1}}{\fwd{w^K_{n}}{h}} \] for
  all $K \in \cC$.  The map on the left is an isomorphism by the
  assumption that $h$ is $p$-excisive, and hence so is the one on the
  right.  Theorem~\ref{thm:n-excisive-equals-Pn-local} (2) then gives
  the desired result.

  (2). If $k$ is any $(n+m)$-excisive map, we have
  \[ \fwd{f \pp w^K_n}{k} = \fwd{f}{\fwd{w^K_n}{k}} \]
  But the map $\fwd{w^K_n}{k}$ is $m$-excisive by (1).

  (3). Again by Theorem~\ref{thm:n-excisive-equals-Pn-local} (2), it
  suffices to check that the map $\fwd{w^K_{p-m}}{\fwd{f}{h}}$ is an
  isomorphism for any $K \in \cC$.  But
  \[ \fwd{w^K_{p-m}}{\fwd{f}{h}} = \fwd{w^K_{p-m} \pp f}{h} \]
  and since $w^K_{p-m} \pp f$ is $p$-excisive by (2), the right
  map is an isomorphism.

  (4). Let $k$ be any $(m+n+1)$-excisive map.  Then since
  \[ \fwd{f \pp g}{k} = \fwd{f}{\fwd{g}{k}} \]
  and the map $\fwd{g}{k}$ is $(m+n+1) - n - 1 = m$ excisive,
  the result follows.
\end{proof}

The compatibility of the Goodwillie tower with the pushout product stated in Theorem~\ref{thm:adjunction-tricks}(4) is what we are really after. It will allow us to prove the Blakers-Massey analogue for the Goodwillie tower. One direct application is
\begin{example}\label{ex:pointed-functor}
Recall that a functor $F$ is $m$-reduced if the map $F\to\term$ is a $P_{m-1}$-equivalence. Let $F$ be $m$-reduced and $G$ be $n$-reduced. Then the map
\[
  (\term\to F)\pp(\term\to G)= (F\vee G\to F\times G)
\]
is a $P_{m+n-1}$-equivalence. Taking the cofiber it follows that $F\wedge G$ is $(m+n-1)$-reduced because, as a left class, $P_n$-equivalences are stable by cobase change. Similarly, the map
\[
   (F\to\term)\pp(G\to\term)= (F\join G\to\term) 
\]
is a $P_{m+n-1}$-equivalence, i.e. $F\join G$ is $(m+n-1)$-reduced.
\end{example}

\subsection{The Blakers-Massey Theorem for the Goodwillie Tower}

\begin{theorem}[Blakers-Massey theorem for Goodwillie Calculus]\label{thm:Goodwillie-Blakers-Massey}
Let
  \[
  \begin{tikzcd}
    F \ar[r, "g"] \ar[d, "f"'] & H \ar[d] \\
    G \ar[r] & K \pomark
  \end{tikzcd}
  \]
be a pushout square of functors. If $f$ is a $P_m$-equivalence and $g$ is a $P_n$-equivalence, then the induced map
  \[ \gap{f}{g} : F \to G \times_K H \]
is a $P_{m + n + 1}$-equivalence.
\end{theorem}

\begin{proof}
If a map $h$ is $k$-excisive then its diagonal $\Delta h$ is also $k$-excisive because $P_k$ is left exact.
Theorem~\ref{thm:adjunction-tricks}(4) then implies that $\Delta f\,\square\,\Delta g$ is a $P_{m+n+1}$-equivalence: $\Delta f\,\square\,\Delta g$ is in the left class of the modality associated to $P_{m+n+1}$. Now we apply the Theorem~\ref{thm:gen-bm} and learn that $\gap{f}{g}$ is in the same left class. Hence, the gap map is a $P_{m+n+1}$-equivalence.
\end{proof}

\begin{theorem}[``Dual'' Blakers-Massey theorem for Goodwillie Calculus]\label{thm:dual-Goodwillie-Blakers-Massey}
Let
  \[
  \begin{tikzcd}
    F \pbmark \ar[r] \ar[d] & H \ar[d, "f"] \\
    G \ar[r, "g"] & K
  \end{tikzcd}
  \]
be a pullback square of functors. If $f$ is a $P_m$-equivalence and $g$ is a $P_n$-equivalence, then the cogap map
  \[ \cogap{f}{g}: G\sqcup_F H \to K\]
is a $P_{m + n + 1}$-equivalence.
\end{theorem}

\begin{proof}
By Theorem~\ref{thm:adjunction-tricks}(4) the map $f\pp g$ is a  $P_{m+n+1}$-equivalence.
By Theorem~\ref{thm:dual-gbm} the same holds for the cogap $\cogap{f}{g}$. 
\end{proof}

\subsection{Delooping theorems}

In this section, we rederive some of the fundamental delooping results
of \cite{G03}.  To begin, let us recall the following definition:

\begin{definition}
  Let $f : X \to Y$ be a map in a topos $\cE$.  We say that $f$ is a
  \emph{principal fibration} if there exists an object $B \in \cE$,
  map $h : Y \to B$ and cover $b : 1 \surj B$ such that the square
  \[
    \begin{tikzcd}
      X \ar[d, "f"'] \ar[r] \pbmark & 1 \ar[d, two heads, "b"]  \\
      Y \ar[r, "h"'] & B
    \end{tikzcd}
  \]
  is cartesian.
\end{definition}

Fix a functor $F:\cC\to\cS$.  It is easily checked that if $F$ is
$k$-excisive, then it is $n$-excisive for any $n \geq k$.  Hence the
universal property of $P_n$ implies that for any $k \leq n$ we have a
canonical map $q_{n,k} : P_n F \to P_k F$. They form the {\it
  Goodwillie tower of $F$}. The map $q_{n,k}$ is a $P_k$-equivalence
by construction.  We will use the abbreviation
\[
  q_n = q_{n,n-1} : P_n F \to P_{n-1}F. 
\]
Now let us denote by $C$ the pushout of the following diagram
\begin{equation}\label{eq:c-square}
  \begin{tikzcd}
    P_n F \ar[r, "q_{n,0}"] \ar[d, "q_nF"'] & P_0 F \ar[d, "c"] \\
    P_{n-1} F \ar[r] & C \pomark
  \end{tikzcd}
\end{equation}
in $[\Cat{C},\Cat{S}]$. We obtain an induced map
$\cogap{q_{n-1, 0}}{\id_{P_0F}} : C \to P_0 F$ so that we may regard
the above diagram as living in the slice category
$[\cC,\cS]\slice{P_0 F}$. Note that $c$ is a $P_{n-1}$-equivalence
because, as a left class of a factorization system,
$P_{n-1}$-equivalences are closed by cobase change. Clearly
$\id_{P_0F}$ is a $P_{n-1}$-equivalence. So $C\to P_0F$ is also a
$P_{n-1}$-equivalence. More vaguely stated, $C$ is pointed and
$n$-reduced relative to the constant functor $P_0F$.

Applying $P_n$ to the square~\eqref{eq:c-square} one obtains the induced square
\begin{equation}\label{eq:pn-square}
\begin{tikzcd}
  P_n F \ar[r] \ar[d, "q_nF"'] & P_0 F \ar[d, "P_nc" two heads] \\
  P_{n-1} F \ar[r] & P_nC 
\end{tikzcd}
\end{equation}
in $\Snexc\slice{P_0 F}$. By Remark~\ref{n-exc-colimits} this is still a pushout in $\Snexc\slice{P_0 F}$, but more is true.

\begin{theorem}\label{thm:delooping}
The square~\eqref{eq:pn-square} is cartesian and the map
  \[ q_n : P_n F \to P_{n-1} F \]
is a principal fibration in the topos $\Snexc\slice{P_0 F}$.
\end{theorem}

\begin{proof}
  To see that the square is cartesian, it suffices to work in the
  ambient topos $[\cC,\cS]$, since the forgetul functors 
  $\Snexc\slice{P_0 F}\to \Snexc \to [\cC,\cS]$
  preserves and reflects pullbacks.
  The map $q_{n,0}$ is a $P_0$-equivalence
  and $q_n$ is a $P_{n-1}$-equivalence.  Hence, applying
  Theorem~\ref{thm:Goodwillie-Blakers-Massey}, we find that the
  cartesian gap map $P_nF \to P_{n-1} F \times_{P_n C} P_0 F$ is a
  $P_n$-equivalence.  But since the source and target of this map are
  $n$-excisive, the gap map is an isomorphism. Thus the
  square~\ref{eq:pn-square} is cartesian as claimed.

  Unwinding the definition of principal fibration given above in the
  slice category $\Snexc\slice{P_0F}$, we find that it remains to
  verify that the map $P_n c$ is a cover.  In fact, it suffices to
  check the statement in the category $\Snexc$. By
  Proposition~\ref{thm:epi-mono-nexc} it suffices to check that the
  map $P_n c$ is a $P_0$-equivalence. But it is even a
  $P_{n-1}$-equivalence since $c$ is, as we have already observed.
\end{proof}

If the category $\cC$ is pointed, there exists a canonical map $F(\term) \to F(X)$ for any $X\in \cC$. This induces maps $P_0F \to F\to P_nF$. 
We define $D_nF$, the {\em $n$-th homogeneous layer of $F$}, by the pullback square
\[
\begin{tikzcd}
D_nF \ar[r]\ar[d] & P_nF\ar[d] \\
P_0F \ar[r] & P_{n-1}F.
\end{tikzcd}
\]

\begin{cor}\label{cor:delooping}
The functor $P_nC$ is a delooping of $D_nF$ in the categories $[\cC,\cS]\slice{P_0F}$ and $\Snexc\slice{P_0F}$ in the sense that $D_nF = \Omega_{P_0F} P_nC$.
\end{cor}

This corollary yields a proof of Goodwillie's delooping result for
homogeneous functors independent of Section 2 of his paper~\cite{G03}.
\begin{proof}
We have the following commutative diagram
\[
\begin{tikzcd}
D_nF \ar[r]\ar[d] \pbmark & P_nF\ar[d] \ar[r]  \pbmark & P_0F\ar[d] \\
P_0F \ar[r] & P_{n-1}F \ar[r] & P_nC.
\end{tikzcd}
\]
where both squares are pullbacks.  We deduce that
$D_nF = \Omega_{P_0F} P_nC$ where $\Omega_{P_0F}$ denotes the loop
functor in the category $[\cC,\cS]\slice{P_0F}$.  Moreover, as every
object in the above diagram is $n$-excisive, we may regard the diagram
as living in the subcategory $\Snexc\slice{P_0F}$, and since the
inclusion $\Snexc\slice{P_0F} \hookrightarrow [\cC,\cS]\slice{P_0F}$
is fully-faithful and preserves limits, the second assertion follows
as well.
\end{proof}

\begin{theorem}\label{thm:segmentspushout}
Consider a cocartesian square
  \[
  \begin{tikzcd}
    F \ar[r, "g"] \ar[d, "f"'] & H \ar[d] \\
    G \ar[r] & K \pomark
  \end{tikzcd}
  \]
in $[\cC,\cS]$ where $f$ and $g$ are $P_n$-equivalences and $F,G$ and $H$ are $(2n+1)$-excisive. Then the induced square
  \[
  \begin{tikzcd}
    F \ar[r, "g"] \ar[d, "f"'] & H \ar[d] \\
    G \ar[r] & P_{2n+1}K
  \end{tikzcd}
  \]
is cartesian in $[\cC,\cS]$.
\end{theorem}

\begin{proof}
By Theorem~\ref{thm:Goodwillie-Blakers-Massey} the gap
\[
   (f,g): F\to G\times_K H 
\]
is a $P_{2n+1}$-equivalence. The comparison map
\[
   G\times_K H\to G\times_{P_{2n+1}K} H 
\]
induced by $K\to P_{2n+1}K$ is also a $P_{2n+1}$-equivalence. So their composition is a $P_{2n+1}$-equivalence between $(2n+1)$-excisive functors. Hence it is an isomorphism.
\end{proof}

\begin{cor}\emph{(Arone-Dwyer-Lesh~\cite[Thm. 4.2]{Arone-Dwyer-Lesh})}
\label{thm:ADL}
For every $n$-reduced functor $F$ the canonical map
\[
   P_{2n-1}F\to\Omega P_{2n-1}\Sigma F 
\]
is an isomorphism. If $F$ is also $(2n-1)$-excisive it is infinitely deloopable.
\end{cor}

The assertion follows from Theorem~\ref{thm:segmentspushout} but we going to give a direct proof. 
\begin{proof}
The isomorphism follows by applying Theorem~\ref{thm:Goodwillie-Blakers-Massey} to the pushout square
\[
  \begin{tikzcd}
  F \ar[r] \ar[d] & \term \ar[d] \\
  \term \ar[r] & \Sigma F . \pomark
  \end{tikzcd} 
\]
Since the class of $P_{n-1}$-equivalence is stable by colimits in $\cE^\to$, the functor $\Sigma F$ is $n$-reduced when $F$ is.
The theorem may then be iterated by taking $P_{2n-1}\Sigma F$ in place of $F$.
\end{proof}

\begin{theorem}
Consider a cartesian square
  \[
  \begin{tikzcd}
    F \ar[r] \ar[d] \pbmark & H \ar[d, "h"] \\
    G \ar[r, "k"] & K 
  \end{tikzcd}
  \]
in $[\cC,\cS]$ where $h$ and $k$ are $P_n$-equivalences and $G,H$ and $K$ are $(2n+1)$-excisive. Then the cogap map $\cogap{h}{k}$ is a $P_{2n+1}$-equivalence and the square is cocartesian in $\Snexc$.
\end{theorem}

\begin{proof}
By Theorem~\ref{thm:dual-Goodwillie-Blakers-Massey} the cogap map
\[
   \cogap{h}{k}:G\sqcup_F H\to K 
\]
is a $P_{2n+1}$-equivalence. Now note that $F$, as a limit of $(2n+1)$-excisive functors, is $(2n+1)$-excisive. Hence $P_{2n+1}(G\sqcup_F H)$ is the pushout in the category $\Snexc$. So the cogap map in $\Snexc$
\[   
   P_{2n+1}(G\sqcup_F H)\to K
\]
is a $P_{2n+1}$-equivalence between $(2n+1)$-excisive functors. Hence it is an isomorphism.
\end{proof}

\appendix
\section{Truncated and connected maps of $n$-excisive functors}
\label{sec:appendix-epis-monos}

Recall that every topos admits a factorization system consisting of
the monomorphisms and covers (whose definition we will recall
momentarily).  The goal of this appendix is to describe this modality
in the topos $[\cC, \cS]\exc{n}$ of $n$-excisive functors.

\begin{definition}
  A map $f:X\to Y$ in a topos $\cE$ a {\it monomorphism} if its diagonal map
  \[ \Delta f: X\to X\times_Y X \]
  is an isomorphism. A map is a {\it cover} if it is
  left orthogonal to every monomorphism.
\end{definition}

\begin{example}\label{exmp:em-spaces}
  In the category of spaces $\cS$, a map is a monomorphism if and
  only if it is the inclusion of a union of path components. The
  covers in spaces are exactly the maps that induce a surjection on
  the set of path components.
\end{example}

Before stating the next result, recall that an object $X$ in a
category $\cC$ is called \emph{discrete} if the space $\cC(K, X)$ is
discrete for every object $K\in \cC$.  Moreover, \cite[Proposition
5.5.6.18]{LurieHT} shows that if the category $\cC$ is presentable
(for example, if $\cC$ is a topos) then the inclusion of the full
subcategory $ \cE_0 \hookrightarrow \cE$ of discrete objects admits a
left adjoint $\tau_0 : \cE \to \cE_0$ which we will refer to as
\emph{$0$-truncation functor}.  A posteriori, we may characterize the
discrete objects $X \in \cE$ as those for which the canonical map
$X \to \tau_0 X$ is an isomorphism.

\begin{example} Here are examples of discrete objects.
\begin{enumerate}
\item If $\cE = \cS$, then $\tau_0$ is the functor sending a space $X$
  to its set $\pi_0 X$ of path components regarded as a discrete
  space.
\item If $\cE = [\cC, \cS]$ it is not hard to see that
  \[ (\tau_0F)(K) = \pi_0(F(K)) \] for every $F\in \cE$ and every
  $K\in \cC$, so that a functor $F$ is discrete if and only if it
  takes values in discrete spaces.
\end{enumerate}
\end{example}

The following folklore proposition asserts that 
monomorphisms and covers in a topos $\cE$ are
essentially determined by their restriction to discrete objects.

\begin{proposition}
  \label{prop:epi-mono-topos}
Let $f : X \to Y$ be a morphism in a topos. 
\begin{enumerate}
  \item[\emph{(1)}] The map $f$ is a monomorphism if and only if $\tau_0 f$ is a monomorphism and the square
    \[
    \begin{tikzcd}
      X \ar[r] \ar[d, "f"', rightarrowtail] \pbmark & \tau_0 X \ar[d, "\tau_0 f", rightarrowtail] \\
      Y \ar[r] & \tau_0 Y
    \end{tikzcd}
    \]
is a pullback.
  \item[\emph{(2)}] The map $f$ is a cover if and only if $\tau_0 f$ is a cover.
\end{enumerate}
\end{proposition}

We will prove the following result which characterizes monomorphisms
and covers in the topos of $n$-excisive functors:

\begin{theorem}
  \label{thm:epi-mono-nexc}
Let $f : F \to G$ be a map in $[\cC, \cS]\exc{n}$. Then:
  \begin{enumerate}
  \item[\emph{(1)}] The map $f$ is monomorphism if and only if $P_0 f$ is a monomorphism and the square
    \[
    \begin{tikzcd}
      F \ar[r] \ar[d, "f"', rightarrowtail] \pbmark & P_0 F \ar[d, "P_0 f", rightarrowtail] \\
      G \ar[r] & P_0 G
    \end{tikzcd}
    \]
is a pullback.
  \item[\emph{(2)}] The map $f$ is a cover if and only if $P_0 f$ is a cover.
  \end{enumerate}
\end{theorem}

\begin{remark}
  \label{rem:locality-of-epi-mono}
  We invite the reader to observe the similarity between 
  Proposition~\ref{prop:epi-mono-topos} and Theorem~\ref{thm:epi-mono-nexc}.
  The category $[\cC, \cS]^{(0)}$ is equivalent to the category
  of spaces and the functor $P_0$ is equivalent to the
  evaluation functor $F\mapsto F(\term)$. It follows
  from the theorem that a map $f:X\to Y$ in $[\cC, \cS]\exc{n}$
  is a cover if and only if the map $f(\term):X(\term)\to Y(\term)$
  is a cover in the category of spaces.
\end{remark}

We begin with some generalities: let us suppose that we have a
left exact localization $P : \cE \to \cF$ with fully faithful
right adjoint $i : \cF \hookrightarrow \cE$.  We will write
$\tau_0^\cE$ and $\tau_0^\cF$ for the $0$-truncation functors
of $\cE$ and $\cF$ respectively.  Now, it follows from the fact
that $P$ preserves colimits that $P$ also commutes with $0$-truncation.
That is
\begin{equation}
  \label{eq:local-trunc}
  P \tau_0^\cE \simeq \tau_0^\cF P
\end{equation}
On the other hand, the inclusion $i$ does \emph{not}, in general,
preserve colimits and hence when we identify $\cF$ with a full
subcategory of $\cE$, we must distinguish between these two distinct
operations.

Specializing to the case at hand, the following notation will be
convenient:

\begin{definition}
In what follows, we write $\tau_0$ for $0$-truncation in $[\cC, \cS]$
and $\tau^{(n)}_0$ for the $0$-truncation functor in the $n$-excisive
localization $[\cC, \cS]\exc{n}$.  
\end{definition}

\begin{remark}
  \label{rem:0-trunc-const}
  The case $n=0$ here merits special attention.  In this case, the
  functor $i : [\cC, \cS]\exc{0} \hookrightarrow [\cC, \cS]$ admits both
  a left and right adjoint given respectively by left and right Kan
  extension along the inclusion of the terminal object
  $\term \hookrightarrow \cC$.  As a consequence, the $0$-truncation
  functors $\tau_0\exc{0}$ and $\tau_0$ \emph{do} coincide on the
  essential image of $i$, which can be identified with the
  \emph{constant} functors.
\end{remark}

\begin{lemma}\label{lem:one-excisive-discrete}
   A discrete functor $F:\cC \to  \cS$ is $1$-excisive if and only it is
  constant. 
\end{lemma}

\begin{proof}
If $F:\cC \to  \cS$ is $1$-excisive, then the following square
is cartesian for every $K \in \cC$:
 \[
  \begin{tikzcd}
    F(K) \ar[r] \ar[d] \pbmark & F(\term) \ar[d, rightarrowtail] \\
    F(\term) \ar[r, rightarrowtail] & F(\Sigma K)
  \end{tikzcd}
  \]
But the map $F(\Sigma K)\to F(\term)$ is a left inverse
of the map $F(\term)\to F(\Sigma K)$, since the map $\Sigma K\to \term$
is a left inverse of the map $\term \to \Sigma K$.
Hence the map $F(\term)\to F(\Sigma K)$ is monic, since $F(\Sigma K)$
is discrete by hypothesis. 
It follows that the map $F(K)\to F(\term)$
is invertible since the square above is cartesian.
\end{proof} 

The preceding lemma allows us to calculate the action of the
$0$-truncation functor $\tau_0\exc{1}$ in the category of $1$-excisive
functors.  The result asserts that the $0$-truncation of a
$1$-excisive $F$ is constant with value the $0$-truncation of the
space $F(\term)$.

\begin{lemma}\label{lem:pi-naught-comm}
For $F \in [\cC, \cS]\exc{1}$, we have
  \[ \tau_0^{(1)} F = \tau_0 P_0 F \]
\end{lemma}

\begin{proof}
  Note that the functor $\tau\exc{1}_0 F$ is both discrete and
  $1$-excisive by definition.  According to
  Lemma~\ref{lem:one-excisive-discrete}, then, it is constant.
  The functor $\tau_0 P_0 F$ is also constant and hence, to show
  that the two agree, it suffices to show they agree after
  evaluation at $\term \in \cC$.  But we have
  \[ \tau_0^{(1)}(F) (\term) = P_0\tau_0^{(1)}(F) (\term) = \tau_0\exc{0} P_0(F) (\term) = \tau_0 P_0(F) (\term) \]
  Where the first equality is by the definition of $P_0$, the
  second is an application of \ref{eq:local-trunc} to the localization
  $P_0 : [\cC, \cS]\exc{1} \to [\cC, \cS]\exc{0}$, and the last
  follows from Remark~\ref{rem:0-trunc-const}.
\end{proof}

% \begin{proof} Recall that the localization functor $P_0:[\cC, \cS]\to [\cC, \cS]\exc{0}$ associates to $F:\cC\to \cS$ the constant functor $P_0F:\cC\to \cS$ with value $F(\term)$.
% The functor $\tau_0 P_0 (F)$ is constant, since it is $0$-excisive.
% The functor $\tau_0^{(1)} F$ is also constant by Lemma \ref{lem:one-excisive-discrete}, since it is 
% $1$-excisive and discrete. Moreover, we have $ P_0 \tau_0^{(1)} (F) \simeq \tau_0 P_0 (F)$ by the case $n=1$ of the identity above. It follows that
% \[ \tau_0^{(1)}(F) (\term) \simeq P_0\tau_0^{(1)}(F) (\term) \simeq \tau_0 P_0(F) (\term) \]
% Thus, $\tau_0^{(1)}(F)\simeq \tau_0 P_0(F)$.
% \end{proof}

\begin{proposition}
  \label{prop:nexc-monos}
  Every $n$-excisive monomorphism is $0$-excisive.
\end{proposition}

\begin{proof}
  The proof is by induction on $n \geq 0$. The case $n=0$ is obvious.
  Let us suppose $n = 1$.
  So let $f : F \to G$ be a $1$-excisive monomorphism.
  The map $P_1f$ is monic, since the functor $P_1$ is a left exact localization.
  The following square is cartesian, since $f$ is 1-excisive.
  \[
    \begin{tikzcd}
      F \ar[r] \ar[d, "f"', rightarrowtail] \pbmark & P_1 F \ar[d, "P_1 f", rightarrowtail] \\
      G \ar[r] & P_1 G 
    \end{tikzcd}
  \]
  Hence it suffices to show that $P_1 f$ is $0$-excisive.
  Hence we may suppose, without loss of generality, that $F$ and $G$ are
  in fact $1$-excisive functors. Consider the cube:
  \[
    \begin{tikzcd}
      F \ar[dd, rightarrowtail] \ar[dr] \ar[rr] & & P_0 F \ar[dd, rightarrowtail] \ar[dr] & \\
      & \tau^{(1)}_0 F \ar[rr, crossing over, "\simeq", near start] & & \tau_0 P_0 F \ar[dd, rightarrowtail] \\
      G \ar[rr] \ar[dr] & & P_0 G \ar[dr] & \\
      & \tau^{(1)}_0 G \ar[rr, "\simeq", near start] \ar[from=uu, crossing over, rightarrowtail] & & \tau_0 P_0 G \\
    \end{tikzcd}
  \]
  Let us show that the back face is a pullback.
All of the vertical maps are monomorphisms since the functors $P_0$, $\tau_0$ and
  $\tau^{(1)}_0$ preserve them.  Both the left and the right face are
  pullbacks by Proposition~\ref{prop:epi-mono-topos}.  By Lemma
  \ref{lem:pi-naught-comm}, the front two horizontal maps are
  in fact isomorphisms, since $F$ and $G$ are 1-excisive. Consequently, the back face is a pullback,
  which says that $f$ is $0$-excisive.

  For the inductive step, let $f : F \to G$ be a $(n+1)$-excisive
  monomorphism.  Note that the functor $\fwd{w_K}{-}$ preserves
  monomorphisms. So, for all $K$, $\fwd{w_K}{f}$ is a
  monomorphism. But it is also $n$-excisive by
  Theorem~\ref{thm:adjunction-tricks}.(1).  By the induction
  hypothesis, it is then $0$-excisive.  This, in turn, shows that $f$
  is $1$-excisive. Then the case $n=1$ implies that $f$ is also
  $0$-excisive.
\end{proof}

\begin{proof}[Proof of Theorem~\emph{\ref{thm:epi-mono-nexc}}]
(1 $\Rightarrow$). This is immediate since $P_0$ preserves monomorphisms and the pullback expresses just the statement that $f$ is $0$-excisive.

  (1 $\Leftarrow$).  Monomorphisms are always stable by pullback.

  (2 $\Rightarrow$). The functor $P_0$ preserves covers because it is a localization.

  (2 $\Leftarrow$).  Note that $f$ is a cover if and only if it is orthogonal
  to every monomorphism in $[\cC, \cS]\exc{n}$.  So let $g : H \to K$ be such a monomorphism
  and consider a lifting problem as follows:
  \[
  \begin{tikzcd}
    F \ar[r] \ar[d, "f"'] & H \ar[d, "g", tail] \\ 
    G \ar[r] & K
  \end{tikzcd}
  \]
Note that since $g$ is a monomorphism, it is enough to
  show that a lift exists, as its uniqueness is automatic.  Now apply
  the functor $P_0$ to obtain
  \[
  \begin{tikzcd}
    P_0 F \ar[r] \ar[d, "P_0 f"', two heads] & P_0 H \ar[d, "P_0 g", rightarrowtail] \\ 
    P_0 G \ar[r] \ar[ur, "\exists !", dotted] & P_0 K
  \end{tikzcd}
  \]
  Observe that the left map is a cover by assumption. Since
  $P_0$ preserves monomorphisms, this square has a unique lift.  But
  now composition of our lift with the map $p_0G : G \to P_0 G$ yields
  the lift shown in the diagram:
  \[
  \begin{tikzcd}
    F \ar[r] \ar[d, "f"'] & H \ar[d, "g", tail] \ar[r] \pbmark & P_0 H \ar[d, "P_0 g", tail] \\ 
    G \ar[r] \ar[urr, "\exists !", dotted, near start] & K \ar[r] & P_0 K
  \end{tikzcd}
  \]
On the other hand, Proposition~\ref{prop:nexc-monos} asserts that the right hand square is a pullback. Hence we have an induced unique lift to the original problem.  This shows that $f$ is a cover.
\end{proof}

%\bibliographystyle{alpha}
%\bibliography{../common/common}

\begin{thebibliography}{{Goo}03}

\bibitem[ABFJ17]{ABFJ:gbm}
Mathieu Anel, Georg Biedermann, Eric Finster, and {Andr\'{e}} Joyal.
\newblock {A generalized Blakers-Massey theorem}.
\newblock \url{https://arxiv.org/abs/1703.09050}, 2017.

\bibitem[ADL08]{Arone-Dwyer-Lesh}
Gregory~Z. Arone, William~G. Dwyer, and Kathryn Lesh.
\newblock Loop structures in {T}aylor towers.
\newblock {\em Algebr. Geom. Topol.}, 8(1):173--210, 2008.

\bibitem[BJM15]{bauer2015cross}
Kristine Bauer, Brenda Johnson, and Randy McCarthy.
\newblock Cross effects and calculus in an unbased setting.
\newblock {\em Transactions of the American Mathematical Society},
  367(9):6671--6718, 2015.

\bibitem[BR14]{BR13}
Georg {Biedermann} and Oliver {R{\"o}ndigs}.
\newblock {Calculus of functors and model categories, II.}
\newblock {\em Algebr. Geom. Topol.}, 14(5):2853--2913, 2014.

\bibitem[{Goo}90]{G90}
Thomas~G. {Goodwillie}.
\newblock {Calculus I: The First Derivative of Pseudoisotropy Theory}.
\newblock {\em K-Theory}, 4, 1990.

\bibitem[{Goo}92]{G92}
Thomas~G. {Goodwillie}.
\newblock {Calculus II: Analytic Functors}.
\newblock {\em K-Theory}, 5, 1992.

\bibitem[{Goo}03]{G03}
Thomas~G. {Goodwillie}.
\newblock {Calculus III: Taylor Series}.
\newblock {\em Geometry and Topology}, 7, October 2003.

\bibitem[{Heu}15]{H15}
Gijs {Heuts}.
\newblock {Goodwillie approximations to higher categories}.
\newblock {\tt arXiv:1510.03304}, October 2015.

\bibitem[Joy08]{JoyalNQC}
Andre Joyal.
\newblock Notes on quasicategories.
\newblock {\url{http://www.math.uchicago.edu/~may/IMA/Joyal.pdf}}, 2008.

\bibitem[Kuh07]{kuhn2007goodwillie}
Nicholas~J Kuhn.
\newblock Goodwillie towers and chromatic homotopy: an overview.
\newblock {\em Proceedings of the Nishida Fest (Kinosaki 2003)}, 10:245--279,
  2007.

\bibitem[Lur09]{LurieHT}
Jacob Lurie.
\newblock {\em Higher topos theory}, volume 170 of {\em Annals of Mathematics
  Studies}.
\newblock Princeton University Press, Princeton, NJ, 2009.

\bibitem[Lur16]{lurie2016higher}
Jacob Lurie.
\newblock Higher algebra.
\newblock \url{http://www.math.harvard.edu/~lurie/}, 2016.

\bibitem[Rez05]{RezkTopos}
Charles Rezk.
\newblock Toposes and homotopy toposes.
\newblock \url{http://www.math.uiuc.edu/~rezk}, 2005.

\bibitem[{Rez}13]{R08}
Charles {Rezk}.
\newblock {A streamlined proof of Goodwillie's $n$-excisive approximation}.
\newblock {\em Algebr. Geom. Topol.}, 13(2):1049--1051, 2013.

\bibitem[Wei95]{weiss1995orthogonal}
Michael Weiss.
\newblock Orthogonal calculus.
\newblock {\em Transactions of the American mathematical society},
  347(10):3743--3796, 1995.

\end{thebibliography}

\end{document}